\documentclass[onefignum,onetabnum]{siamart190516}

\usepackage{amsmath}
\usepackage{amssymb}
\usepackage{mathrsfs}
\usepackage{amsopn}
\usepackage{bm}
\usepackage{lipsum}
\usepackage{amsfonts}
\usepackage{graphicx}
\usepackage{epstopdf}
\usepackage{algorithmic}
 \usepackage{booktabs}
\ifpdf
  \DeclareGraphicsExtensions{.eps,.pdf,.png,.jpg}
\else
  \DeclareGraphicsExtensions{.eps}
\fi

\newsiamremark{example}{Example}
\newsiamremark{remark}{Remark}
\newsiamremark{hypothesis}{Hypothesis}
\crefname{hypothesis}{Hypothesis}{Hypotheses}
\newsiamthm{claim}{Claim}

\DeclareMathOperator*{\Null}{null}
\DeclareMathOperator*{\Range}{range}
\DeclareMathOperator*{\rank}{rank}

\let\s=\scriptscriptstyle

\headers{Convergence analysis of inexact two-grid methods}{Xuefeng Xu and Chen-Song Zhang}

\title{{Convergence analysis of inexact two-grid methods: A theoretical framework}\thanks{Submitted to the editors DATE.
\funding{This paper is based on Xu's Ph.D.~thesis~\cite{XXF-thesis} at the Academy of Mathematics and Systems Science, Chinese Academy of Sciences. Zhang was partially supported by the National Key R\&D Program of China (2020YFA0711900, 2020YFA0711904), the National Science Foundation of China (11971472), and the Key Research Program of Frontier Sciences of CAS.}}}

\author{Xuefeng Xu\thanks{Corresponding author.~Department of Mathematics, Purdue University, West Lafayette, IN 47907, USA (\email{xuxuefeng@lsec.cc.ac.cn}, \email{xu1412@purdue.edu}).}
\and Chen-Song Zhang\thanks{LSEC \& NCMIS, Academy of Mathematics and Systems Science, Chinese Academy of Sciences, Beijing 100190, China, and School of Mathematical Sciences, University of Chinese Academy of Sciences, Beijing 100049, China (\email{zhangcs@lsec.cc.ac.cn}).}}

\ifpdf
\hypersetup{
  pdftitle={Convergence analysis of inexact two-grid methods},
  pdfauthor={Xuefeng Xu and Chen-Song Zhang}
}
\fi

\begin{document}

\maketitle

\begin{abstract}
Multigrid is one of the most efficient methods for solving large-scale linear systems that arise from discretized partial differential equations. As a foundation for multigrid analysis, two-grid theory plays an important role in motivating and analyzing multigrid algorithms. For symmetric positive definite problems, the convergence theory of two-grid methods with exact solution of the Galerkin coarse-grid system is mature, and the convergence factor of exact two-grid methods can be characterized by an identity. Compared with the exact case, the convergence theory of inexact two-grid methods (i.e., the coarse-grid system is solved approximately) is of more practical significance, while it is still less developed in the literature (one reason is that the error propagation matrix of inexact coarse-grid correction is not a projection). In this paper, we develop a theoretical framework for the convergence analysis of inexact two-grid methods. More specifically, we present two-sided bounds for the energy norm of the error propagation matrix of inexact two-grid methods, from which one can readily obtain the identity for exact two-grid convergence. As an application, we establish a unified convergence theory for multigrid methods, which allows the coarsest-grid system to be solved approximately.
\end{abstract}



\section{Introduction}

Multigrid is a powerful solver, with linear or near-linear computational complexity, for a large class of linear systems arising from discretized partial differential equations; see, e.g.,~\cite{Briggs2000,Trottenberg2001,Vassilevski2008}. The idea of multigrid originated with Fedorenko in the 1960s~\cite{Fedorenko1962,Fedorenko1964}, while it received scant attention until the work of Brandt in the 1970s~\cite{Brandt1973,Brandt1977}. Some fundamental elements for the convergence analysis of multigrid methods are attributed to Hackbusch~\cite{Hackbusch1980,Hackbusch1981}. Other representative works on the early development of multigrid methods can be found in~\cite{Hackbusch1985,Yserentant1993,Briggs2000,Trottenberg2001} and the references therein. Since the early 1980s, multigrid has been extensively studied and applied in scientific and engineering computing; see, e.g.,~\cite{Trottenberg2001,Vassilevski2008}.

The foundation of multigrid methods is a two-grid scheme, which combines two complementary processes: \textit{smoothing} (or \textit{local relaxation}) and \textit{coarse-grid correction}. The smoothing process is typically a simple iterative method, such as the (weighted) Jacobi and Gauss--Seidel iterations. In general, these classical methods are efficient at eliminating high-frequency (i.e., oscillatory) error, while low-frequency (i.e., smooth) error components cannot be eliminated effectively~\cite{Briggs2000,Trottenberg2001}. To remedy this defect, a coarse-grid correction strategy is used in the two-grid scheme: the low-frequency error can be further reduced by solving a coarse-grid system. The coarse-grid correction process involves two intergrid operators that transfer information between fine- and coarse-grids: one is a \textit{restriction} matrix that restricts the residual formed on a fine-grid to a coarser-grid; the other is a \textit{prolongation} (or \textit{interpolation}) matrix $P\in\mathbb{R}^{n\times n_{\rm c}}$ ($n$ and $n_{\rm c}$ are the numbers of fine and coarse variables, respectively) with full column rank that extends the correction computed on the coarse-grid to the fine-grid. Typically, the restriction matrix is taken to be $P^{T}$, the transpose of $P$, as considered in this paper. The so-called \textit{Galerkin coarse-grid matrix} is defined as $A_{\rm c}:=P^{T}AP\in\mathbb{R}^{n_{\rm c}\times n_{\rm c}}$, which gives a coarse representation of the fine-grid matrix $A\in\mathbb{R}^{n\times n}$ ($A$ is assumed to be symmetric positive definite (SPD)).

Regarding two-grid analysis, most previous works (see, e.g.,~\cite{Falgout2005,Zikatanov2008,Notay2015,XZ2017,Brannick2018}) focus on exact two-grid methods, with some exceptions like~\cite{Notay2007,Vassilevski2008}. An identity has been established to characterize the convergence factor of exact two-grid methods~\cite{XZ2002,Falgout2005}. In practice, however, it is often too costly to solve the Galerkin coarse-grid system exactly, especially when its size is large. Instead, one may solve the coarse-grid system approximately as long as the convergence speed is satisfactory. A recursive call (e.g., the V- and W-cycles) of two-grid procedure yields a multigrid method, which can be regarded as an inexact two-grid scheme. It is well known that two-grid convergence is sufficient to assess the W-cycle multigrid convergence; see, e.g.,~\cite{Hackbusch1985,Trottenberg2001}. Based on the idea of hierarchical basis~\cite{Bank1988} and the minimization property of Schur complements (see, e.g.,~\cite[Theorem~3.8]{Axelsson1994}), Notay~\cite{Notay2007} derived an upper bound for the convergence factor of inexact two-grid methods. With this estimate, Notay~\cite[Theorem~3.1]{Notay2007} showed that, if the convergence factor of exact two-grid method at any level is uniformly bounded by $\sigma<1/2$, then the convergence factor of the corresponding W-cycle multigrid method is uniformly bounded by $\sigma/(1-\sigma)$.

Besides theoretical considerations, two-grid theory can be used to guide the design of multigrid algorithms. The implementation of a multigrid scheme on large-scale parallel machines is still a challenging topic, especially in the era of exascale computing. For instance, stencil sizes (the number of nonzero entries in a row) of the Galerkin coarse-grid matrices tend to increase further down in the multilevel hierarchy of algebraic multigrid methods~\cite{Brandt1985,Brandt1986,Ruge1987}, which will increase the cost of communication. As problem size increases and the number of levels grows, the overall efficiency of parallel algebraic multigrid methods may decrease dramatically~\cite{Falgout2014}. Motivated by the inexact two-grid theory in~\cite{Notay2007}, Falgout and Schroder~\cite{Falgout2014} proposed a non-Galerkin coarsening strategy to improve the parallel efficiency of algebraic multigrid algorithms. Some other sparse approximations to $A_{\rm c}$ can be found, e.g., in~\cite{Brandt2000,Sterck2006,Sterck2008}.

Algebraic multigrid constructs the coarsening process in a purely algebraic manner (that is, the explicit knowledge of geometric properties is not required), which has been widely applied in scientific and engineering problems associated with complex domains, unstructured grids, jump coefficients, etc; see, e.g.,~\cite{Vassilevski2008,XZ2017}. As stated in~\cite{Notay2007}, it is possible to prove optimal convergence properties of multigrid methods via some smoothing and approximation properties or via the theory of subspace correction methods; see, e.g.,~\cite{Braess1982,Braess1983,Hackbusch1985,McCormick1985,Mandel1988,Xu1992,Yserentant1993,Oswald1994}. However, convergence bounds derived by these approaches do not, in general, give satisfactory predictions of actual convergence speed~\cite[page~96]{Trottenberg2001}. Moreover, for algebraic multigrid methods, it may be difficult to check some required assumptions~\cite{Notay2007}. In fact, two-grid analysis is still a main strategy for assessing and analyzing algebraic multigrid methods~\cite{MacLachlan2014,Notay2015}.

In this paper, we develop a theoretical framework for the convergence analysis of inexact two-grid methods, in which the Galerkin coarse-grid matrix $A_{\rm c}$ is replaced by a general SPD matrix $B_{\rm c}\in\mathbb{R}^{n_{\rm c}\times n_{\rm c}}$. More precisely, we present lower and upper bounds for the energy norm of the error propagation matrix of inexact two-grid methods, from which one can readily get the identity for exact two-grid convergence. The new upper bounds are sharper than the existing one in~\cite{Notay2007} (see~\cref{comp-Notay}). As an application of the framework, we establish a unified convergence theory for multigrid methods, in which the coarsest-grid system is not required to be solved exactly.

The rest of this paper is organized as follows. In~\cref{sec:pre}, we review some fundamental properties of two-grid methods and an elegant identity for the convergence factor of exact two-grid methods. In~\cref{sec:iTG}, we present a theoretical framework for the convergence analysis of inexact two-grid methods. In~\cref{sec:MG}, we establish a unified convergence theory for multigrid methods based on the proposed framework. In~\cref{sec:con}, we give some concluding remarks.

\section{Preliminaries} \label{sec:pre}

In this section, we review some useful properties of two-grid methods, which play a fundamental role in the convergence analysis of inexact two-grid methods. For convenience, we list some notation used in the subsequent discussions.

\begin{itemize}

\item[--] $I_{n}$ denotes the $n\times n$ identity matrix (or $I$ when its size is clear from context).

\item[--] $\lambda_{\min}(\cdot)$, $\lambda_{\min}^{+}(\cdot)$, and $\lambda_{\max}(\cdot)$ denote the smallest eigenvalue, the smallest positive eigenvalue, and the largest eigenvalue of a matrix, respectively.

\item[--] $\lambda(\cdot)$ denotes the spectrum of a matrix.

\item[--] $\rho(\cdot)$ denotes the spectral radius of a matrix.

\item[--] $\langle\cdot,\cdot\rangle$ denotes the standard Euclidean inner product of two vectors.

\item[--] $\|\cdot\|_{2}$ denotes the spectral norm of a matrix.

\item[--] $\|\cdot\|_{A}$ denotes the energy norm induced by an SPD matrix $A\in\mathbb{R}^{n\times n}$: for any $\mathbf{v}\in\mathbb{R}^{n}$, $\|\mathbf{v}\|_{A}=\langle A\mathbf{v},\mathbf{v}\rangle^{\frac{1}{2}}$; for any $B\in\mathbb{R}^{n\times n}$, $\|B\|_{A}=\max\limits_{\mathbf{v}\in\mathbb{R}^{n}\backslash\{0\}}\frac{\|B\mathbf{v}\|_{A}}{\|\mathbf{v}\|_{A}}$.

\item[--] $\kappa_{A}(\cdot)$ denotes the condition number, with respect to $\|\cdot\|_{A}$, of a matrix.

\end{itemize}

\subsection{Two-grid methods} \label{subsec:TG}

Consider solving the linear system
\begin{equation}\label{system}
A\mathbf{u}=\mathbf{f},
\end{equation}
where $A\in\mathbb{R}^{n\times n}$ is SPD, $\mathbf{u}\in\mathbb{R}^{n}$, and $\mathbf{f}\in\mathbb{R}^{n}$. Given an initial guess $\mathbf{u}^{(0)}\in\mathbb{R}^{n}$ and a nonsingular smoother $M\in\mathbb{R}^{n\times n}$, we perform the following iteration:
\begin{equation}\label{relax}
\mathbf{u}^{(k+1)}=\mathbf{u}^{(k)}+M^{-1}\big(\mathbf{f}-A\mathbf{u}^{(k)}\big) \quad k=0,1,\ldots
\end{equation}
From~\cref{relax}, we have
\begin{displaymath}
\mathbf{u}-\mathbf{u}^{(k+1)}=(I-M^{-1}A)\big(\mathbf{u}-\mathbf{u}^{(k)}\big),
\end{displaymath}
which leads to
\begin{displaymath}
\|\mathbf{u}-\mathbf{u}^{(k)}\|_{A}\leq\|I-M^{-1}A\|_{A}^{k}\|\mathbf{u}-\mathbf{u}^{(0)}\|_{A}.
\end{displaymath}
For any initial guess $\mathbf{u}^{(0)}$, if $\|I-M^{-1}A\|_{A}<1$, then
\begin{displaymath}
\lim_{k\rightarrow+\infty}\|\mathbf{u}-\mathbf{u}^{(k)}\|_{A}=0.
\end{displaymath}
Since
\begin{displaymath}
\|(I-M^{-1}A)\mathbf{v}\|_{A}^{2}=\|\mathbf{v}\|_{A}^{2}-\langle(M+M^{T}-A)M^{-1}A\mathbf{v},\,M^{-1}A\mathbf{v}\rangle \quad \forall\,\mathbf{v}\in\mathbb{R}^{n},
\end{displaymath}
a sufficient and necessary condition for the iteration~\cref{relax} to be \textit{$A$-convergent}, that is, $\|I-M^{-1}A\|_{A}<1$, is that $M+M^{T}-A$ is SPD.

For an $A$-convergent smoother $M$, we define two \textit{symmetrized} variants:
\begin{subequations}
\begin{align}
\overline{M}&:=M(M+M^{T}-A)^{-1}M^{T},\label{barM}\\
\widetilde{M}&:=M^{T}(M+M^{T}-A)^{-1}M.\label{tildM}
\end{align}
\end{subequations}
It is easy to check that
\begin{subequations}
\begin{align}
I-\overline{M}^{-1}A&=(I-M^{-T}A)(I-M^{-1}A),\label{rela-barM}\\
I-\widetilde{M}^{-1}A&=(I-M^{-1}A)(I-M^{-T}A),\label{rela-tildM}
\end{align}
\end{subequations}
from which one can easily deduce that both $\overline{M}-A$ and $\widetilde{M}-A$ are symmetric positive semidefinite (SPSD).

Usually, the iteration~\cref{relax} can only eliminate high-frequency error effectively. To further reduce the remaining low-frequency modes, a coarse-grid correction strategy is used in two-grid scheme. Let $P\in\mathbb{R}^{n\times n_{\rm c}}$ be a prolongation (or interpolation) matrix of rank $n_{\rm c}$, where $n_{\rm c} \ (<n)$ is the number of coarse variables. The Galerkin coarse-grid matrix takes the form $A_{\rm c}=P^{T}AP\in\mathbb{R}^{n_{\rm c}\times n_{\rm c}}$. Let $\mathbf{u}^{(\ell)}\in\mathbb{R}^{n}$ be an approximation to $\mathbf{u}\equiv A^{-1}\mathbf{f}$, e.g., $\mathbf{u}^{(\ell)}$ is generated from~\cref{relax}. The (exact) coarse-grid correction can be described as follows:
\begin{equation}\label{correction}
\mathbf{u}^{(\ell+1)}=\mathbf{u}^{(\ell)}+PA_{\rm c}^{-1}P^{T}\big(\mathbf{f}-A\mathbf{u}^{(\ell)}\big).
\end{equation}
Let
\begin{equation}\label{piA}
\varPi_{A}=PA_{\rm c}^{-1}P^{T}A.
\end{equation}
Then
\begin{displaymath}
\mathbf{u}-\mathbf{u}^{(\ell+1)}=(I-\varPi_{A})\big(\mathbf{u}-\mathbf{u}^{(\ell)}\big).
\end{displaymath}
Note that $I-\varPi_{A}$ is an $A$-orthogonal projection along (or parallel to) $\Range(P)$ onto $\Null(P^{T}A)$. Thus,
\begin{displaymath}
(I-\varPi_{A})\mathbf{e}=0 \quad\forall\,\mathbf{e}\in\Range(P),
\end{displaymath}
which suggests that $I-\varPi_{A}$ can remove the error components contained in the coarse space $\Range(P)$. That is, an efficient coarse-grid correction will be achieved if $\Range(P)$ can cover most of the low-frequency error.

With the iterations~\cref{relax} and~\cref{correction}, a symmetric two-grid scheme for solving~\cref{system} can be described by~\cref{alg:TG}. If the SPD coarse-grid matrix $B_{\rm c}$ in~\cref{alg:TG} is taken to be $A_{\rm c}$, then the algorithm is called an \textit{exact} two-grid method; otherwise, it is called an \textit{inexact} two-grid method.

\begin{algorithm}

\caption{\textbf{Two-grid method}}

\label{alg:TG}

\begin{algorithmic}[1]

\STATE Presmoothing: $\mathbf{u}^{(1)}\gets\mathbf{u}^{(0)}+M^{-1}\big(\mathbf{f}-A\mathbf{u}^{(0)}\big)$ \hfill\COMMENT{$M+M^{T}-A\in\mathbb{R}^{n\times n}$ is SPD}

\smallskip

\STATE Restriction: $\mathbf{r}_{\rm c}\gets P^{T}\big(\mathbf{f}-A\mathbf{u}^{(1)}\big)$ \hfill\COMMENT{$P\in\mathbb{R}^{n\times n_{\rm c}}$ has full column rank}

\smallskip

\STATE Coarse-grid correction: $\mathbf{e}_{\rm c}\gets B_{\rm c}^{-1}\mathbf{r}_{\rm c}$ \hfill\COMMENT{$B_{\rm c}\in\mathbb{R}^{n_{\rm c}\times n_{\rm c}}$ is SPD}

\smallskip

\STATE Prolongation: $\mathbf{u}^{(2)}\gets\mathbf{u}^{(1)}+P\mathbf{e}_{\rm c}$

\smallskip

\STATE Postsmoothing: $\mathbf{u}_{\rm ITG}\gets\mathbf{u}^{(2)}+M^{-T}\big(\mathbf{f}-A\mathbf{u}^{(2)}\big)$

\smallskip

\end{algorithmic}

\end{algorithm}

From~\cref{alg:TG}, we have
\begin{displaymath}
\mathbf{u}-\mathbf{u}_{\rm ITG}=E_{\rm ITG}\big(\mathbf{u}-\mathbf{u}^{(0)}\big),
\end{displaymath}
where
\begin{equation}\label{tE_TG1}
E_{\rm ITG}=(I-M^{-T}A)(I-PB_{\rm c}^{-1}P^{T}A)(I-M^{-1}A).
\end{equation}
It is referred to as the \textit{iteration matrix} (or \textit{error propagation matrix}) of~\cref{alg:TG}, which can be expressed as
\begin{equation}\label{tE_TG2}
E_{\rm ITG}=I-B_{\rm ITG}^{-1}A
\end{equation}
with
\begin{equation}\label{inv_tB_TG}
B_{\rm ITG}^{-1}=\overline{M}^{-1}+(I-M^{-T}A)PB_{\rm c}^{-1}P^{T}(I-AM^{-1}).
\end{equation}
Since $\overline{M}$ and $B_{\rm c}$ are SPD, we deduce from~\cref{inv_tB_TG} that $B_{\rm ITG}$ is an SPD matrix, which is called the \textit{inexact two-grid preconditioner}. In view of~\cref{tE_TG2}, we have
\begin{equation}\label{norm_tE_TG}
\|E_{\rm ITG}\|_{A}=\rho(E_{\rm ITG})=\max\big\{\lambda_{\max}\big(B_{\rm ITG}^{-1}A\big)-1,\,1-\lambda_{\min}\big(B_{\rm ITG}^{-1}A\big)\big\},
\end{equation}
which is referred to as the \textit{convergence factor} of~\cref{alg:TG}.

\subsection{Convergence of exact two-grid methods} \label{subsec:XZ}

The convergence properties of~\cref{alg:TG} with $B_{\rm c}=A_{\rm c}$ have been well studied by the multigrid community. For its algebraic analysis, we refer to~\cite{Vassilevski2008,MacLachlan2014,Notay2015} and the references therein.

Denote the iteration matrix of~\cref{alg:TG} with $B_{\rm c}=A_{\rm c}$ by $E_{\rm TG}$. Then
\begin{equation}\label{E_TG1}
E_{\rm TG}=(I-M^{-T}A)(I-\varPi_{A})(I-M^{-1}A),
\end{equation}
where $\varPi_{A}$ is given by~\cref{piA}. Similarly, $E_{\rm TG}$ can be expressed as
\begin{equation}\label{E_TG2}
E_{\rm TG}=I-B_{\rm TG}^{-1}A
\end{equation}
with
\begin{equation}\label{invB_TG}
B_{\rm TG}^{-1}=\overline{M}^{-1}+(I-M^{-T}A)PA_{\rm c}^{-1}P^{T}(I-AM^{-1}).
\end{equation}
The SPD matrix $B_{\rm TG}$ is called the \textit{exact two-grid preconditioner}.

The following theorem provides an identity for the convergence factor $\|E_{\rm TG}\|_{A}$~\cite[Theorem~4.3]{Falgout2005}, which is a two-level version of the XZ-identity~\cite{XZ2002,Zikatanov2008}.

\begin{theorem}
Let $\widetilde{M}$ be defined by~\cref{tildM}, and let
\begin{equation}\label{piM}
\varPi_{\widetilde{M}}=P(P^{T}\widetilde{M}P)^{-1}P^{T}\widetilde{M}.
\end{equation}
Then, the convergence factor of~\cref{alg:TG} with $B_{\rm c}=A_{\rm c}$ can be characterized as
\begin{equation}\label{XZ}
\|E_{\rm TG}\|_{A}=1-\frac{1}{K_{\rm TG}},
\end{equation}
where
\begin{equation}\label{K_TG}
K_{\rm TG}=\max_{\mathbf{v}\in\mathbb{R}^{n}\backslash\{0\}}\frac{\big\|(I-\varPi_{\widetilde{M}})\mathbf{v}\big\|_{\widetilde{M}}^{2}}{\|\mathbf{v}\|_{A}^{2}}.
\end{equation}
\end{theorem}

\begin{remark}
The matrix $\varPi_{\widetilde{M}}$ given by~\cref{piM} is an $\widetilde{M}$-orthogonal projection onto $\Range(P)$. That is, $\varPi_{\widetilde{M}}^{2}=\varPi_{\widetilde{M}}$, $\Range(\varPi_{\widetilde{M}})=\Range(P)$, and $\varPi_{\widetilde{M}}$ is self-adjoint with respect to the inner product $\langle\cdot,\cdot\rangle_{\widetilde{M}}:=\langle\widetilde{M}\cdot,\cdot\rangle$.
\end{remark}

\begin{remark}\label{rmk:cond-numb}
The expression~\cref{E_TG1} implies that $A^{\frac{1}{2}}E_{\rm TG}A^{-\frac{1}{2}}$ is an SPSD matrix with smallest eigenvalue $0$. Since
\begin{displaymath}
A^{\frac{1}{2}}E_{\rm TG}A^{-\frac{1}{2}}=I-A^{\frac{1}{2}}B_{\rm TG}^{-1}A^{\frac{1}{2}},
\end{displaymath}
we get that $B_{\rm TG}-A$ is also SPSD and $\lambda_{\max}\big(B_{\rm TG}^{-1}A\big)=1$. Due to
\begin{displaymath}
1-\frac{1}{K_{\rm TG}}=\|E_{\rm TG}\|_{A}=\rho(E_{\rm TG})=\lambda_{\max}(E_{\rm TG})=1-\lambda_{\min}\big(B_{\rm TG}^{-1}A\big),
\end{displaymath}
it follows that
\begin{displaymath}
\lambda_{\min}\big(B_{\rm TG}^{-1}A\big)=\frac{1}{K_{\rm TG}}.
\end{displaymath}
As a result, we have
\begin{displaymath}
K_{\rm TG}=\frac{\lambda_{\max}\big(B_{\rm TG}^{-1}A\big)}{\lambda_{\min}\big(B_{\rm TG}^{-1}A\big)}.
\end{displaymath}
This shows that $K_{\rm TG}$ is the corresponding condition number when~\cref{alg:TG} with $B_{\rm c}=A_{\rm c}$ is treated as a preconditioning method.
\end{remark}

As is well known, the aim of two-grid methods is to balance the interplay between smoother and coarse space (or interpolation). For a fixed smoother $M$ (e.g., the Jacobi or Gauss--Seidel type), an \textit{optimal} interpolation can be obtained by minimizing $K_{\rm TG}$. In practice, however, it is often too costly to compute the optimal interpolation, because it requires the explicit knowledge of eigenvectors corresponding to small eigenvalues of the generalized eigenvalue problem $A\mathbf{x}=\lambda\widetilde{M}\mathbf{x}$; see~\cite{XZ2017,Brannick2018} for details. To find a cheap alternative to the optimal interpolation, one may minimize a suitable upper bound for $K_{\rm TG}$.

Let $R$ be an $n_{\rm c}\times n$ matrix with the property $RP=I_{n_{\rm c}}$, and let $Q=PR$. Clearly, $Q\in\mathbb{R}^{n\times n}$ is a projection onto $\Range(P)$. In light of~\cref{K_TG}, we have
\begin{displaymath}
K_{\rm TG}=\max_{\mathbf{v}\in\mathbb{R}^{n}\backslash\{0\}}\min_{\mathbf{v}_{\rm c}\in\mathbb{R}^{n_{\rm c}}}\frac{\|\mathbf{v}-P\mathbf{v}_{\rm c}\|_{\widetilde{M}}^{2}}{\|\mathbf{v}\|_{A}^{2}}\leq\max_{\mathbf{v}\in\mathbb{R}^{n}\backslash\{0\}}\frac{\|(I-Q)\mathbf{v}\|_{\widetilde{M}}^{2}}{\|\mathbf{v}\|_{A}^{2}}=:K,
\end{displaymath}
which, together with~\cref{XZ}, yields
\begin{displaymath}
\|E_{\rm TG}\|_{A}\leq 1-\frac{1}{K}.
\end{displaymath}
By minimizing $K$ over all interpolations, one can obtain an \textit{ideal} interpolation~\cite{Falgout2004,XXF2018}, which gives a strategy for designing an interpolation with sparse or simple structure; see, e.g.,~\cite{Manteuffel2017,Manteuffel2018,XXF2018,Manteuffel2019}. In particular, if $R=(P^{T}\widetilde{M}P)^{-1}P^{T}\widetilde{M}$, then $K=K_{\rm TG}$; see~\cite{XXF2018} for a quantitative relation between $K$ and $K_{\rm TG}$. Hence, the ideal interpolation can be viewed as a generalization of the optimal one.

\section{Convergence of inexact two-grid methods} \label{sec:iTG}

In this section, we develop a general framework for the convergence analysis of~\cref{alg:TG}. More specifically, lower and upper bounds for the convergence factor of~\cref{alg:TG} are established.

According to~\cref{norm_tE_TG}, the main task of estimating $\|E_{\rm ITG}\|_{A}$ is to bound the extreme eigenvalues of $B_{\rm ITG}^{-1}A$. It was proved by Notay~\cite[Theorem~2.2]{Notay2007} that
\begin{subequations}
\begin{align}
&\lambda_{\max}\big(B_{\rm ITG}^{-1}A\big)\leq\max\big\{1,\,\lambda_{\max}(B_{\rm c}^{-1}A_{\rm c})\big\}\lambda_{\max}\big(B_{\rm TG}^{-1}A\big),\label{Notay-b1}\\
&\lambda_{\min}\big(B_{\rm ITG}^{-1}A\big)\geq\min\big\{1,\,\lambda_{\min}(B_{\rm c}^{-1}A_{\rm c})\big\}\lambda_{\min}\big(B_{\rm TG}^{-1}A\big),\label{Notay-b2}
\end{align}
\end{subequations}
which, together with~\cref{norm_tE_TG} and~\cref{rmk:cond-numb}, yield
\begin{equation}\label{Notay1}
\|E_{\rm ITG}\|_{A}\leq\max\Bigg\{1-\frac{\min\big\{1,\,\lambda_{\min}(B_{\rm c}^{-1}A_{\rm c})\big\}}{K_{\rm TG}},\,\max\big\{1,\,\lambda_{\max}(B_{\rm c}^{-1}A_{\rm c})\big\}-1\Bigg\}.
\end{equation}
Consider a special case that $B_{\rm c}=\alpha I_{n_{\rm c}}$ with $\alpha>0$. In this case, the iteration matrix $E_{\rm ITG}$ takes the form
\begin{displaymath}
E_{\rm ITG}=(I-M^{-T}A)\bigg(I-\frac{1}{\alpha}PP^{T}A\bigg)(I-M^{-1}A).
\end{displaymath}
Obviously, \cref{alg:TG} tends to an algorithm with only two smoothing steps when $\alpha\rightarrow+\infty$. It is easy to check that the convergence factor of the limiting algorithm is $1-\lambda_{\min}(\widetilde{M}^{-1}A)<1$. However, the estimate~\cref{Notay1} gives nothing but a trivial upper bound $1$, from which one cannot explicitly determine whether the limiting algorithm is convergent. This suggests that the estimate~\cref{Notay1} is not sharp in some situations.

In what follows, we establish a new convergence theory for~\cref{alg:TG} based on some technical eigenvalue identities and the well-known Weyl's theorem.

We first give several important eigenvalue identities, which will be frequently used in the subsequent analysis.

\begin{lemma}
The extreme eigenvalues of $(I-\widetilde{M}^{-1}A)(I-\varPi_{A})$ and $(I-\widetilde{M}^{-1}A)\varPi_{A}$ have the following properties:
\begin{subequations}
\begin{align}
&\lambda_{\min}\big((I-\widetilde{M}^{-1}A)(I-\varPi_{A})\big)=0,\label{eig1.1}\\
&\lambda_{\max}\big((I-\widetilde{M}^{-1}A)(I-\varPi_{A})\big)=1-\frac{1}{K_{\rm TG}},\label{eig1.2}\\
&\lambda_{\min}\big((I-\widetilde{M}^{-1}A)\varPi_{A}\big)=0,\label{eig1.3}\\
&\lambda_{\max}\big((I-\widetilde{M}^{-1}A)\varPi_{A}\big)=1-\lambda_{\min}^{+}(\widetilde{M}^{-1}A\varPi_{A}).\label{eig1.4}
\end{align}
\end{subequations}
\end{lemma}

\begin{proof}
Since $\varPi_{A}^{2}=\varPi_{A}$ and $\rank(\varPi_{A})=n_{\rm c}$, there exists a nonsingular matrix $X\in\mathbb{R}^{n\times n}$ such that
\begin{displaymath}
X^{-1}\varPi_{A}X=\begin{pmatrix}
I_{n_{\rm c}} & 0 \\
0 & 0
\end{pmatrix}.
\end{displaymath}
Let
\begin{displaymath}
X^{-1}\widetilde{M}^{-1}AX=\begin{pmatrix}
\widehat{X}_{11} & \widehat{X}_{12} \\
\widehat{X}_{21} & \widehat{X}_{22}
\end{pmatrix},
\end{displaymath}
where $\widehat{X}_{ij}\in\mathbb{R}^{m_{i}\times m_{j}}$ with $m_{1}=n_{\rm c}$ and $m_{2}=n-n_{\rm c}$. Then
\begin{subequations}
\begin{align}
&X^{-1}\big(\widetilde{M}^{-1}A(I-\varPi_{A})+\varPi_{A}\big)X=\begin{pmatrix}
I_{n_{\rm c}} & \widehat{X}_{12} \\
0 & \widehat{X}_{22}
\end{pmatrix},\label{exp-1}\\
&X^{-1}(I-\widetilde{M}^{-1}A)(I-\varPi_{A})X=\begin{pmatrix}
0 & -\widehat{X}_{12} \\
0 & I-\widehat{X}_{22}
\end{pmatrix},\label{exp-2}\\
&X^{-1}(I-\widetilde{M}^{-1}A)\varPi_{A}X=\begin{pmatrix}
I-\widehat{X}_{11} & 0 \\
-\widehat{X}_{21} & 0
\end{pmatrix},\label{exp-3}\\
&X^{-1}\widetilde{M}^{-1}A\varPi_{A}X=\begin{pmatrix}
\widehat{X}_{11} & 0 \\
\widehat{X}_{21} & 0
\end{pmatrix}.\label{exp-4}
\end{align}
\end{subequations}

Using~\cref{rela-barM} and~\cref{invB_TG}, we obtain
\begin{align*}
B_{\rm TG}^{-1}A&=\overline{M}^{-1}A+(I-M^{-T}A)\varPi_{A}(I-M^{-1}A)\\
&=I-(I-M^{-T}A)(I-\varPi_{A})(I-M^{-1}A),
\end{align*}
which, together with~\cref{rela-tildM}, yields
\begin{align*}
\lambda\big(B_{\rm TG}^{-1}A\big)&=\lambda\big(I-(I-M^{-1}A)(I-M^{-T}A)(I-\varPi_{A})\big)\\
&=\lambda\big(I-(I-\widetilde{M}^{-1}A)(I-\varPi_{A})\big)\\
&=\lambda\big(\widetilde{M}^{-1}A(I-\varPi_{A})+\varPi_{A}\big).
\end{align*}
According to~\cref{rmk:cond-numb} and~\cref{exp-1}, we deduce that
\begin{displaymath}
\lambda(\widehat{X}_{22})\subset\bigg[\frac{1}{K_{\rm TG}},\,1\bigg] \quad \text{and} \quad \lambda_{\min}(\widehat{X}_{22})=\frac{1}{K_{\rm TG}}.
\end{displaymath}
Then, by~\cref{exp-2}, we have
\begin{align*}
\lambda_{\min}\big((I-\widetilde{M}^{-1}A)(I-\varPi_{A})\big)&=\min\big\{0,\,1-\lambda_{\max}(\widehat{X}_{22})\big\}=0,\\
\lambda_{\max}\big((I-\widetilde{M}^{-1}A)(I-\varPi_{A})\big)&=\max\big\{0,\,1-\lambda_{\min}(\widehat{X}_{22})\big\}=1-\frac{1}{K_{\rm TG}}.
\end{align*}

Since $\widetilde{M}-A$ is SPSD and $\varPi_{A}=PA_{\rm c}^{-1}P^{T}A$, it follows that
\begin{displaymath}
\lambda(\widetilde{M}^{-1}A\varPi_{A})\subset[0,1],
\end{displaymath}
which, combined with~\cref{exp-4}, yields $\lambda(\widehat{X}_{11})\subset[0,1]$. In light of~\cref{exp-3}, we have
\begin{align}
\lambda_{\min}\big((I-\widetilde{M}^{-1}A)\varPi_{A}\big)&=\min\big\{0,\,1-\lambda_{\max}(\widehat{X}_{11})\big\}=0,\notag\\
\lambda_{\max}\big((I-\widetilde{M}^{-1}A)\varPi_{A}\big)&=\max\big\{0,\,1-\lambda_{\min}(\widehat{X}_{11})\big\}=1-\lambda_{\min}(\widehat{X}_{11}).\label{tMAPi1}
\end{align}
Note that both $A^{-1}-\widetilde{M}^{-1}$ and
\begin{displaymath}
(A^{-1}-\widetilde{M}^{-1})^{\frac{1}{2}}A(A^{-1}-\widetilde{M}^{-1})^{\frac{1}{2}}-(A^{-1}-\widetilde{M}^{-1})^{\frac{1}{2}}A\varPi_{A}(A^{-1}-\widetilde{M}^{-1})^{\frac{1}{2}}
\end{displaymath}
are SPSD. We then have
\begin{align*}
\lambda_{\max}\big((I-\widetilde{M}^{-1}A)\varPi_{A}\big)&=\lambda_{\max}\big((A^{-1}-\widetilde{M}^{-1})^{\frac{1}{2}}A\varPi_{A}(A^{-1}-\widetilde{M}^{-1})^{\frac{1}{2}}\big)\\
&\leq\lambda_{\max}\big((A^{-1}-\widetilde{M}^{-1})^{\frac{1}{2}}A(A^{-1}-\widetilde{M}^{-1})^{\frac{1}{2}}\big)\\
&=1-\lambda_{\min}(\widetilde{M}^{-1}A)<1.
\end{align*}
The above inequality, together with~\cref{tMAPi1}, leads to $\lambda_{\min}(\widehat{X}_{11})>0$. Thus,
\begin{displaymath}
\lambda_{\max}\big((I-\widetilde{M}^{-1}A)\varPi_{A}\big)=1-\lambda_{\min}(\widehat{X}_{11})=1-\lambda_{\min}^{+}(\widetilde{M}^{-1}A\varPi_{A}).
\end{displaymath}
This completes the proof.
\end{proof}

Let $S_{1}\in\mathbb{R}^{n\times n}$ and $S_{2}\in\mathbb{R}^{n\times n}$ be symmetric matrices. Denote the spectra of $S_{1}$, $S_{2}$, and $S_{1}+S_{2}$ by $\{\lambda_{i}(S_{1})\}_{i=1}^{n}$, $\{\lambda_{i}(S_{2})\}_{i=1}^{n}$, and $\{\lambda_{i}(S_{1}+S_{2})\}_{i=1}^{n}$, respectively. For each $k=1,\ldots,n$, the Weyl's theorem (see, e.g.,~\cite[Theorem~4.3.1]{Horn2013}) states that
\begin{equation}\label{Weyl}
\lambda_{k-j+1}(S_{1})+\lambda_{j}(S_{2})\leq\lambda_{k}(S_{1}+S_{2})\leq\lambda_{k+\ell}(S_{1})+\lambda_{n-\ell}(S_{2})
\end{equation}
for all $j=1,\ldots,k$ and $\ell=0,\ldots,n-k$, where $\lambda_{i}(\cdot)$ denotes the $i$-th smallest eigenvalue of a matrix. In particular, one has
\begin{subequations}
\begin{align}
\lambda_{\min}(S_{1}+S_{2})&\geq\lambda_{\min}(S_{1})+\lambda_{\min}(S_{2}),\label{Weyl-min-low}\\
\lambda_{\min}(S_{1}+S_{2})&\leq\min\big\{\lambda_{\min}(S_{1})+\lambda_{\max}(S_{2}),\,\lambda_{\max}(S_{1})+\lambda_{\min}(S_{2})\big\},\label{Weyl-min-up}\\
\lambda_{\max}(S_{1}+S_{2})&\geq\max\big\{\lambda_{\max}(S_{1})+\lambda_{\min}(S_{2}),\,\lambda_{\min}(S_{1})+\lambda_{\max}(S_{2})\big\},\label{Weyl-max-low}\\
\lambda_{\max}(S_{1}+S_{2})&\leq\lambda_{\max}(S_{1})+\lambda_{\max}(S_{2}).\label{Weyl-max-up}
\end{align}
\end{subequations}

It is worth noting that the Weyl's theorem can also be \textit{applied} to the nonsymmetric matrix $(I-\widetilde{M}^{-1}A)(I-r\varPi_{A})$ with parameter $r$. Indeed, $(I-\widetilde{M}^{-1}A)(I-r\varPi_{A})$ has the same spectrum as the symmetric matrix $(A^{-1}-\widetilde{M}^{-1})^{\frac{1}{2}}A(I-r\varPi_{A})(A^{-1}-\widetilde{M}^{-1})^{\frac{1}{2}}$. One can first apply the Weyl's theorem to the symmetric one, and then transform the result into a form related to $I-\widetilde{M}^{-1}A$, $(I-\widetilde{M}^{-1}A)\varPi_{A}$, or $(I-\widetilde{M}^{-1}A)(I-\varPi_{A})$. For example, if $r\geq 0$, we get from~\cref{Weyl-max-up} that
\begin{align*}
\lambda_{\max}\big((I-\widetilde{M}^{-1}A)(I-r\varPi_{A})\big)&=\lambda_{\max}\big((A^{-1}-\widetilde{M}^{-1})^{\frac{1}{2}}A(I-r\varPi_{A})(A^{-1}-\widetilde{M}^{-1})^{\frac{1}{2}}\big)\\
&\leq\lambda_{\max}\big((A^{-1}-\widetilde{M}^{-1})^{\frac{1}{2}}A(A^{-1}-\widetilde{M}^{-1})^{\frac{1}{2}}\big)\\
&\quad-r\lambda_{\min}\big((A^{-1}-\widetilde{M}^{-1})^{\frac{1}{2}}A\varPi_{A}(A^{-1}-\widetilde{M}^{-1})^{\frac{1}{2}}\big)\\
&=\lambda_{\max}(I-\widetilde{M}^{-1}A)-r\lambda_{\min}\big((I-\widetilde{M}^{-1}A)\varPi_{A}\big).
\end{align*}
For brevity, such a trick will be \textit{implicitly} used in the subsequent discussions.

We are now in a position to present a new convergence theory for~\cref{alg:TG}.

\begin{theorem}\label{thm:iTG}
Let
\begin{equation}\label{r12}
r_{1}=\lambda_{\min}(B_{\rm c}^{-1}A_{\rm c}) \quad \text{and} \quad r_{2}=\lambda_{\max}(B_{\rm c}^{-1}A_{\rm c}).
\end{equation}
Under the assumptions of~\cref{alg:TG}, the convergence factor $\|E_{\rm ITG}\|_{A}$ satisfies the following estimates.

{\rm (i)} If $r_{2}\leq 1$, then
\begin{equation}\label{iTG1.1}
\mathscr{L}_{1}\leq\|E_{\rm ITG}\|_{A}\leq\mathscr{U}_{1},
\end{equation}
where
\begin{align*}
\mathscr{L}_{1}&=1-\min\bigg\{\frac{1}{K_{\rm TG}},\,\lambda_{\min}(\widetilde{M}^{-1}A)+r_{2}\big(1-\lambda_{\min}^{+}(\widetilde{M}^{-1}A\varPi_{A})\big)\bigg\},\\
\mathscr{U}_{1}&=1-\frac{r_{1}}{K_{\rm TG}}-(1-r_{1})\lambda_{\min}(\widetilde{M}^{-1}A).
\end{align*}

{\rm (ii)} If $r_{1}\leq 1<r_{2}$, then
\begin{equation}\label{iTG1.2}
\mathscr{L}_{2}\leq\|E_{\rm ITG}\|_{A}\leq\max\big\{\mathscr{U}_{1},\,\mathscr{U}_{2}\big\},
\end{equation}
where
\begin{align*}
\mathscr{L}_{2}&=1-\min\bigg\{\lambda_{\max}(\widetilde{M}^{-1}A),\,\frac{r_{2}}{K_{\rm TG}}-(r_{2}-1)\lambda_{\min}(\widetilde{M}^{-1}A)\bigg\},\\
\mathscr{U}_{2}&=(r_{2}-1)\big(1-\lambda_{\min}^{+}(\widetilde{M}^{-1}A\varPi_{A})\big).
\end{align*}

{\rm (iii)} If $1<r_{1}$, then
\begin{equation}\label{iTG1.3}
\max\big\{\mathscr{L}_{2},\,\mathscr{L}_{3}\big\}\leq\|E_{\rm ITG}\|_{A}\leq\mathscr{U}_{3},
\end{equation}
where
\begin{align*}
\mathscr{L}_{3}&=r_{1}-1-\min\Big\{r_{1}\lambda_{\min}^{+}(\widetilde{M}^{-1}A\varPi_{A})-\lambda_{\min}(\widetilde{M}^{-1}A),\,(r_{1}-1)\lambda_{\max}(\widetilde{M}^{-1}A)\Big\},\\
\mathscr{U}_{3}&=\max\bigg\{1-\frac{1}{K_{\rm TG}},\,(r_{2}-1)\big(1-\lambda_{\min}^{+}(\widetilde{M}^{-1}A\varPi_{A})\big)\bigg\}.
\end{align*}
\end{theorem}

\begin{proof}
By~\cref{tE_TG1} and~\cref{tE_TG2}, we have
\begin{displaymath}
B_{\rm ITG}^{-1}A=I-(I-M^{-T}A)(I-PB_{\rm c}^{-1}P^{T}A)(I-M^{-1}A).
\end{displaymath}
Then
\begin{align*}
\lambda\big(B_{\rm ITG}^{-1}A\big)&=\lambda\big(I-(I-M^{-1}A)(I-M^{-T}A)(I-PB_{\rm c}^{-1}P^{T}A)\big)\\
&=\lambda\big(I-(I-\widetilde{M}^{-1}A)(I-PB_{\rm c}^{-1}P^{T}A)\big),
\end{align*}
which yields
\begin{align*}
\lambda_{\max}\big(B_{\rm ITG}^{-1}A\big)&=1-\lambda_{\min}\big((I-\widetilde{M}^{-1}A)(I-PB_{\rm c}^{-1}P^{T}A)\big),\\
\lambda_{\min}\big(B_{\rm ITG}^{-1}A\big)&=1-\lambda_{\max}\big((I-\widetilde{M}^{-1}A)(I-PB_{\rm c}^{-1}P^{T}A)\big).
\end{align*}
Note that $(I-\widetilde{M}^{-1}A)(I-PB_{\rm c}^{-1}P^{T}A)$ has the same eigenvalues as the symmetric matrix $(A^{-1}-\widetilde{M}^{-1})^{\frac{1}{2}}A(I-PB_{\rm c}^{-1}P^{T}A)(A^{-1}-\widetilde{M}^{-1})^{\frac{1}{2}}$. In view of~\cref{r12}, we have
\begin{subequations}
\begin{align}
-s_{1}&\leq\lambda_{\max}\big(B_{\rm ITG}^{-1}A\big)-1\leq-s_{2},\label{max-low-up}\\
t_{2}&\leq 1-\lambda_{\min}\big(B_{\rm ITG}^{-1}A\big)\leq t_{1},\label{min-low-up}
\end{align}
\end{subequations}
where
\begin{align*}
s_{k}&=\lambda_{\min}\big((I-\widetilde{M}^{-1}A)(I-r_{k}\varPi_{A})\big),\\
t_{k}&=\lambda_{\max}\big((I-\widetilde{M}^{-1}A)(I-r_{k}\varPi_{A})\big).
\end{align*}
According to~\cref{norm_tE_TG}, \cref{max-low-up}, and~\cref{min-low-up}, we deduce that
\begin{equation}\label{two-side}
\max\{-s_{1},\,t_{2}\}\leq\|E_{\rm ITG}\|_{A}\leq\max\{-s_{2},\,t_{1}\}.
\end{equation}

Next, we are devoted to establishing the upper bounds for $s_{1}$ and $t_{1}$, as well as the lower bounds for $s_{2}$ and $t_{2}$. The remainder of this proof is divided into three parts corresponding to the cases $r_{2}\leq 1$, $r_{1}\leq 1<r_{2}$, and $1<r_{1}$.

\textbf{Case 1:} $r_{2}\leq 1$. By~\cref{Weyl-min-up}, we have that
\begin{align*}
s_{1}&=\lambda_{\min}\big((I-\widetilde{M}^{-1}A)\big(I-\varPi_{A}+(1-r_{1})\varPi_{A}\big)\big)\\
&\leq\lambda_{\min}\big((I-\widetilde{M}^{-1}A)(I-\varPi_{A})\big)+(1-r_{1})\lambda_{\max}\big((I-\widetilde{M}^{-1}A)\varPi_{A}\big)\\
&=1-r_{1}-(1-r_{1})\lambda_{\min}^{+}(\widetilde{M}^{-1}A\varPi_{A})
\end{align*}
and
\begin{align*}
s_{1}&=\lambda_{\min}\big((I-\widetilde{M}^{-1}A)(I-r_{1}\varPi_{A})\big)\\
&\leq\lambda_{\min}(I-\widetilde{M}^{-1}A)-r_{1}\lambda_{\min}\big((I-\widetilde{M}^{-1}A)\varPi_{A}\big)\\
&=1-\lambda_{\max}(\widetilde{M}^{-1}A),
\end{align*}
where we have used the facts~\cref{eig1.1}, \cref{eig1.3}, and~\cref{eig1.4}. Then
\begin{equation}\label{s1-up1}
s_{1}\leq 1-\max\Big\{r_{1}+(1-r_{1})\lambda_{\min}^{+}(\widetilde{M}^{-1}A\varPi_{A}),\,\lambda_{\max}(\widetilde{M}^{-1}A)\Big\}.
\end{equation}
By~\cref{Weyl-min-low}, we have
\begin{align*}
s_{2}&=\lambda_{\min}\big((I-\widetilde{M}^{-1}A)\big((1-r_{2})I+r_{2}(I-\varPi_{A})\big)\big)\\
&\geq(1-r_{2})\lambda_{\min}(I-\widetilde{M}^{-1}A)+r_{2}\lambda_{\min}\big((I-\widetilde{M}^{-1}A)(I-\varPi_{A})\big),
\end{align*}
which, together with~\cref{eig1.1}, yields
\begin{equation}\label{s2-low1}
s_{2}\geq(1-r_{2})\big(1-\lambda_{\max}(\widetilde{M}^{-1}A)\big).
\end{equation}

Using~\cref{Weyl-max-up}, we obtain
\begin{align*}
t_{1}&=\lambda_{\max}\big((I-\widetilde{M}^{-1}A)\big((1-r_{1})I+r_{1}(I-\varPi_{A})\big)\big)\\
&\leq(1-r_{1})\lambda_{\max}(I-\widetilde{M}^{-1}A)+r_{1}\lambda_{\max}\big((I-\widetilde{M}^{-1}A)(I-\varPi_{A})\big).
\end{align*}
The above inequality, combined with~\cref{eig1.2}, yields
\begin{equation}\label{t1-up1}
t_{1}\leq 1-\frac{r_{1}}{K_{\rm TG}}-(1-r_{1})\lambda_{\min}(\widetilde{M}^{-1}A).
\end{equation}
By~\cref{eig1.2}--\cref{eig1.4} and~\cref{Weyl-max-low}, we have that
\begin{align*}
t_{2}&=\lambda_{\max}\big((I-\widetilde{M}^{-1}A)\big(I-\varPi_{A}+(1-r_{2})\varPi_{A}\big)\big)\\
&\geq\lambda_{\max}\big((I-\widetilde{M}^{-1}A)(I-\varPi_{A})\big)+(1-r_{2})\lambda_{\min}\big((I-\widetilde{M}^{-1}A)\varPi_{A}\big)\\
&=1-\frac{1}{K_{\rm TG}}
\end{align*}
and
\begin{align*}
t_{2}&=\lambda_{\max}\big((I-\widetilde{M}^{-1}A)(I-r_{2}\varPi_{A})\big)\\
&\geq\lambda_{\max}(I-\widetilde{M}^{-1}A)-r_{2}\lambda_{\max}\big((I-\widetilde{M}^{-1}A)\varPi_{A}\big)\\
&=1-\lambda_{\min}(\widetilde{M}^{-1}A)-r_{2}\big(1-\lambda_{\min}^{+}(\widetilde{M}^{-1}A\varPi_{A})\big).
\end{align*}
We then have
\begin{equation}\label{t2-low1}
t_{2}\geq 1-\min\bigg\{\frac{1}{K_{\rm TG}},\,\lambda_{\min}(\widetilde{M}^{-1}A)+r_{2}\big(1-\lambda_{\min}^{+}(\widetilde{M}^{-1}A\varPi_{A})\big)\bigg\}.
\end{equation}
Combining~\cref{two-side}--\cref{t2-low1}, we can arrive at the estimate~\cref{iTG1.1} immediately.

\textbf{Case 2:} $r_{1}\leq 1<r_{2}$. Note that the inequalities~\cref{s1-up1} and~\cref{t1-up1} still hold due to $r_{1}\leq 1$. We next focus on the lower bounds for $s_{2}$ and $t_{2}$. By~\cref{Weyl-min-low}, we have
\begin{align*}
s_{2}&=\lambda_{\min}\big((I-\widetilde{M}^{-1}A)\big(I-\varPi_{A}+(1-r_{2})\varPi_{A}\big)\big)\\
&\geq\lambda_{\min}\big((I-\widetilde{M}^{-1}A)(I-\varPi_{A})\big)+(1-r_{2})\lambda_{\max}\big((I-\widetilde{M}^{-1}A)\varPi_{A}\big),
\end{align*}
which, together with~\cref{eig1.1} and~\cref{eig1.4}, yields
\begin{equation}\label{s2-low2}
s_{2}\geq(1-r_{2})\big(1-\lambda_{\min}^{+}(\widetilde{M}^{-1}A\varPi_{A})\big).
\end{equation}
In light of~\cref{eig1.2}, \cref{eig1.3}, and~\cref{Weyl-max-low}, we have that
\begin{align*}
t_{2}&=\lambda_{\max}\big((I-\widetilde{M}^{-1}A)(I-r_{2}\varPi_{A})\big)\\
&\geq\lambda_{\min}(I-\widetilde{M}^{-1}A)-r_{2}\lambda_{\min}\big((I-\widetilde{M}^{-1}A)\varPi_{A}\big)\\
&=1-\lambda_{\max}(\widetilde{M}^{-1}A)
\end{align*}
and
\begin{align*}
t_{2}&=\lambda_{\max}\big((I-\widetilde{M}^{-1}A)\big((1-r_{2})I+r_{2}(I-\varPi_{A})\big)\big)\\
&\geq(1-r_{2})\lambda_{\max}(I-\widetilde{M}^{-1}A)+r_{2}\lambda_{\max}\big((I-\widetilde{M}^{-1}A)(I-\varPi_{A})\big)\\
&=1-\frac{r_{2}}{K_{\rm TG}}+(r_{2}-1)\lambda_{\min}(\widetilde{M}^{-1}A).
\end{align*}
Hence,
\begin{equation}\label{t2-low2}
t_{2}\geq 1-\min\bigg\{\lambda_{\max}(\widetilde{M}^{-1}A),\,\frac{r_{2}}{K_{\rm TG}}-(r_{2}-1)\lambda_{\min}(\widetilde{M}^{-1}A)\bigg\}.
\end{equation}
The estimate~\cref{iTG1.2} then follows by combining~\cref{two-side}, \cref{s1-up1}, \cref{t1-up1}, \cref{s2-low2}, and~\cref{t2-low2}.

\textbf{Case 3:} $1<r_{1}$. In this case, the estimates~\cref{s2-low2} and~\cref{t2-low2} still hold. We then consider the upper bounds for $s_{1}$ and $t_{1}$. Using~\cref{eig1.1}, \cref{eig1.4}, and~\cref{Weyl-min-up}, we get that
\begin{align*}
s_{1}&=\lambda_{\min}\big((I-\widetilde{M}^{-1}A)(I-r_{1}\varPi_{A})\big)\\
&\leq\lambda_{\max}(I-\widetilde{M}^{-1}A)-r_{1}\lambda_{\max}\big((I-\widetilde{M}^{-1}A)\varPi_{A}\big)\\
&=1-\lambda_{\min}(\widetilde{M}^{-1}A)-r_{1}\big(1-\lambda_{\min}^{+}(\widetilde{M}^{-1}A\varPi_{A})\big)
\end{align*}
and
\begin{align*}
s_{1}&=\lambda_{\min}\big((I-\widetilde{M}^{-1}A)\big((1-r_{1})I+r_{1}(I-\varPi_{A})\big)\big)\\
&\leq(1-r_{1})\lambda_{\min}(I-\widetilde{M}^{-1}A)+r_{1}\lambda_{\min}\big((I-\widetilde{M}^{-1}A)(I-\varPi_{A})\big)\\
&=1-r_{1}+(r_{1}-1)\lambda_{\max}(\widetilde{M}^{-1}A).
\end{align*}
Hence,
\begin{equation}\label{s1-up2}
s_{1}\leq 1-r_{1}+\min\Big\{r_{1}\lambda_{\min}^{+}(\widetilde{M}^{-1}A\varPi_{A})-\lambda_{\min}(\widetilde{M}^{-1}A),\,(r_{1}-1)\lambda_{\max}(\widetilde{M}^{-1}A)\Big\}.
\end{equation}
By~\cref{Weyl-max-up}, we have
\begin{align*}
t_{1}&=\lambda_{\max}\big((I-\widetilde{M}^{-1}A)\big(I-\varPi_{A}+(1-r_{1})\varPi_{A}\big)\big)\\
&\leq\lambda_{\max}\big((I-\widetilde{M}^{-1}A)(I-\varPi_{A})\big)+(1-r_{1})\lambda_{\min}\big((I-\widetilde{M}^{-1}A)\varPi_{A}\big),
\end{align*}
which, combined with~\cref{eig1.2} and~\cref{eig1.3}, gives
\begin{equation}\label{t1-up2}
t_{1}\leq 1-\frac{1}{K_{\rm TG}}.
\end{equation}
In light of~\cref{two-side} and~\cref{s2-low2}--\cref{t1-up2}, we conclude that the estimate~\cref{iTG1.3} is valid.
\end{proof}

\begin{remark}
In particular, if $B_{\rm c}=A_{\rm c}$, then the lower and upper bounds in~\cref{iTG1.1} become
\begin{align*}
\mathscr{L}_{1}&=1-\min\bigg\{\frac{1}{K_{\rm TG}},\,1+\lambda_{\min}(\widetilde{M}^{-1}A)-\lambda_{\min}^{+}(\widetilde{M}^{-1}A\varPi_{A})\bigg\},\\
\mathscr{U}_{1}&=1-\frac{1}{K_{\rm TG}}.
\end{align*}
By~\cref{eig1.2}, \cref{eig1.4}, and~\cref{Weyl-max-up}, we have
\begin{align*}
2-\frac{1}{K_{\rm TG}}-\lambda_{\min}^{+}(\widetilde{M}^{-1}A\varPi_{A})&\geq\lambda_{\max}\big((I-\widetilde{M}^{-1}A)(I-\varPi_{A})+(I-\widetilde{M}^{-1}A)\varPi_{A}\big)\\
&=\lambda_{\max}(I-\widetilde{M}^{-1}A)\\
&=1-\lambda_{\min}(\widetilde{M}^{-1}A),
\end{align*}
which yields
\begin{displaymath}
1+\lambda_{\min}(\widetilde{M}^{-1}A)-\lambda_{\min}^{+}(\widetilde{M}^{-1}A\varPi_{A})\geq\frac{1}{K_{\rm TG}}.
\end{displaymath}
Then
\begin{displaymath}
\mathscr{L}_{1}=1-\frac{1}{K_{\rm TG}}.
\end{displaymath}
Hence, the estimate~\cref{iTG1.1} will reduce to the identity~\cref{XZ} when $B_{\rm c}=A_{\rm c}$.
\end{remark}

\begin{remark}\label{comp-Notay}
With the notation in~\cref{r12}, the estimate~\cref{Notay1} reads
\begin{equation}\label{Notay2}
\|E_{\rm ITG}\|_{A}\leq\begin{cases}
1-\frac{r_{1}}{K_{\rm TG}} & \text{if $r_{2}\leq 1$},\\[2pt]
\max\big\{1-\frac{r_{1}}{K_{\rm TG}},\,r_{2}-1\big\} & \text{if $r_{1}\leq 1<r_{2}$},\\[2pt]
\max\big\{1-\frac{1}{K_{\rm TG}},\,r_{2}-1\big\} & \text{if $1<r_{1}$}.
\end{cases}
\end{equation}
It is easy to see that the upper bounds in~\cref{iTG1.1}--\cref{iTG1.3} are smaller than that in~\cref{Notay2}. On the other hand, if $B_{\rm c}=\alpha I_{n_{\rm c}}$ with $\alpha>0$, then
\begin{displaymath}
\lim_{\alpha\rightarrow+\infty}r_{1}=\lim_{\alpha\rightarrow+\infty}\frac{\lambda_{\min}(A_{\rm c})}{\alpha}=0 \quad \text{and} \quad \lim_{\alpha\rightarrow+\infty}r_{2}=\lim_{\alpha\rightarrow+\infty}\frac{\lambda_{\max}(A_{\rm c})}{\alpha}=0.
\end{displaymath}
One can readily check that both $\mathscr{L}_{1}$ and $\mathscr{U}_{1}$ tend to $1-\lambda_{\min}(\widetilde{M}^{-1}A)$ as $\alpha\rightarrow+\infty$, which is exactly the convergence factor of the limiting algorithm. That is, the estimate~\cref{iTG1.1} has fixed the defect of~\cref{Notay1} indicated at the outset of this section. Besides improved upper bounds, \cref{thm:iTG} provides new lower bounds for $\|E_{\rm ITG}\|_{A}$, which give necessary conditions for a fast convergence speed.
\end{remark}

As mentioned earlier, the Galerkin coarse-grid matrix may affect the parallel efficiency of algebraic multigrid algorithms. To improve the parallel performance, Falgout and Schroder~\cite{Falgout2014} proposed a non-Galerkin coarsening strategy, which is motivated by the following result. Define
\begin{equation}\label{cond-FS}
\theta:=\big\|I-A_{\rm c}^{-1}B_{\rm c}\big\|_{2},
\end{equation}
where $B_{\rm c}\in\mathbb{R}^{n_{\rm c}\times n_{\rm c}}$ is a general SPD approximation to $A_{\rm c}$. If $\theta<1$, then
\begin{equation}\label{FS1}
\kappa_{A}\big(B_{\rm ITG}^{-1}A\big)\leq\frac{1+\theta}{1-\theta}\kappa_{A}\big(B_{\rm TG}^{-1}A\big)=\frac{1+\theta}{1-\theta}K_{\rm TG}
\end{equation}
and
\begin{equation}\label{FS2}
\|E_{\rm ITG}\|_{A}\leq\max\bigg\{\frac{\theta}{1-\theta},\,1-\frac{1}{(1+\theta)K_{\rm TG}}\bigg\}.
\end{equation}

The definition~\cref{cond-FS} implies that
\begin{displaymath}
\theta\geq\rho\big(I-A_{\rm c}^{-1}B_{\rm c}\big)=\max\big\{\lambda_{\max}\big(A_{\rm c}^{-1}B_{\rm c}\big)-1,\,1-\lambda_{\min}\big(A_{\rm c}^{-1}B_{\rm c}\big)\big\},
\end{displaymath}
and hence
\begin{displaymath}
1-\theta\leq\lambda_{\min}\big(A_{\rm c}^{-1}B_{\rm c}\big)\leq\lambda_{\max}\big(A_{\rm c}^{-1}B_{\rm c}\big)\leq 1+\theta.
\end{displaymath}
With the notation in~\cref{r12}, we have
\begin{displaymath}
\frac{1}{1+\theta}\leq r_{1}\leq r_{2}\leq\frac{1}{1-\theta},
\end{displaymath}
which contains the following three cases:
\begin{align*}
&\texttt{C}_\texttt{1}: \ \frac{1}{1+\theta}\leq r_{1}\leq r_{2}\leq 1;\\
&\texttt{C}_\texttt{2}: \ \frac{1}{1+\theta}\leq r_{1}\leq 1<r_{2}\leq\frac{1}{1-\theta};\\
&\texttt{C}_\texttt{3}: \ 1<r_{1}\leq r_{2}\leq\frac{1}{1-\theta}.
\end{align*}
From~\cref{max-low-up} and~\cref{min-low-up}, we have
\begin{displaymath}
\lambda\big(B_{\rm ITG}^{-1}A\big)\subset[1-t_{1},\,1-s_{2}]\subset(0,\,+\infty),
\end{displaymath}
where we have used the facts~\cref{t1-up1} and~\cref{t1-up2}. Then
\begin{displaymath}
\kappa_{A}\big(B_{\rm ITG}^{-1}A\big)=\frac{\lambda_{\max}\big(B_{\rm ITG}^{-1}A\big)}{\lambda_{\min}\big(B_{\rm ITG}^{-1}A\big)}\leq\frac{1-s_{2}}{1-t_{1}}.
\end{displaymath}
According to~\cref{s2-low1}, \cref{t1-up1}, \cref{s2-low2}, and~\cref{t1-up2}, we deduce that
\begin{equation}\label{imp-FS1}
\kappa_{A}\big(B_{\rm ITG}^{-1}A\big)\leq\begin{cases}
\frac{1-\theta}{1+\theta K_{\rm TG}\lambda_{\min}(\widetilde{M}^{-1}A)}\cdot\frac{1+\theta}{1-\theta}K_{\rm TG} & \text{if $\texttt{C}_\texttt{1}$ holds},\\[7pt]
\frac{1-\theta\lambda_{\min}^{+}(\widetilde{M}^{-1}A\varPi_{A})}{1+\theta K_{\rm TG}\lambda_{\min}(\widetilde{M}^{-1}A)}\cdot\frac{1+\theta}{1-\theta}K_{\rm TG} & \text{if $\texttt{C}_\texttt{2}$ holds},\\[7pt]
\frac{1-\theta\lambda_{\min}^{+}(\widetilde{M}^{-1}A\varPi_{A})}{1+\theta}\cdot\frac{1+\theta}{1-\theta}K_{\rm TG} & \text{if $\texttt{C}_\texttt{3}$ holds}.
\end{cases}
\end{equation}
Furthermore, using~\cref{iTG1.1}--\cref{iTG1.3}, we obtain that
\begin{equation}\label{imp-FS2}
\|E_{\rm ITG}\|_{A}\leq\begin{cases}
1-\frac{1+\theta K_{\rm TG}\lambda_{\min}(\widetilde{M}^{-1}A)}{(1+\theta)K_{\rm TG}} & \text{if $\texttt{C}_\texttt{1}$ holds},\\[6pt]
\max\Big\{1-\frac{1+\theta K_{\rm TG}\lambda_{\min}(\widetilde{M}^{-1}A)}{(1+\theta)K_{\rm TG}},\,\frac{\theta-\theta\lambda_{\min}^{+}(\widetilde{M}^{-1}A\varPi_{A})}{1-\theta}\Big\} & \text{if $\texttt{C}_\texttt{2}$ holds},\\[6pt]
\max\Big\{1-\frac{1}{K_{\rm TG}},\,\frac{\theta-\theta\lambda_{\min}^{+}(\widetilde{M}^{-1}A\varPi_{A})}{1-\theta}\Big\} & \text{if $\texttt{C}_\texttt{3}$ holds}.
\end{cases}
\end{equation}
It is easy to see that the estimates~\cref{imp-FS1} and~\cref{imp-FS2} are sharper than~\cref{FS1} and~\cref{FS2}, respectively.

\section{An application of inexact two-grid theory} \label{sec:MG}

In practice, it is often too costly to solve the Galerkin coarse-grid system exactly, especially when its size is large. Instead, without essential loss of convergence speed, one may replace the coarse-grid matrix by a suitable approximation. A natural way to obtain such an approximation is to apply~\cref{alg:TG} recursively in the coarse-grid correction steps. To describe the resulting (multigrid) algorithm conveniently, we give some notation and assumptions.

\begin{itemize}

\item The algorithm involves $L+1$ levels with indices $0,\ldots,L$, where $0$ and $L$ correspond to the coarsest- and finest-levels, respectively.

\item $n_{k}$ denotes the number of unknowns at level $k$ ($n=n_{L}>n_{L-1}>\cdots>n_{0}$).

\item For each $k=1,\ldots,L$, $P_{k}\in\mathbb{R}^{n_{k}\times n_{k-1}}$ denotes a prolongation matrix from level $k-1$ to level $k$, and ${\rm rank}(P_{k})=n_{k-1}$.

\item Let $A_{L}=A$. For each $k=0,\ldots,L-1$, $A_{k}:=P_{k+1}^{T}A_{k+1}P_{k+1}$ denotes the Galerkin coarse-grid matrix at level $k$.

\item Let $\hat{A}_{0}\in\mathbb{R}^{n_{0}\times n_{0}}$ be an SPD approximation to $A_{0}$, and let $\hat{A}_{0}-A_{0}$ be SPSD.

\item For each $k=1,\ldots,L$, $M_{k}\in\mathbb{R}^{n_{k}\times n_{k}}$ denotes a nonsingular smoother at level $k$ with $M_{k}+M_{k}^{T}-A_{k}$ being SPD (or, equivalently, $\|I-M_{k}^{-1}A_{k}\|_{A_{k}}<1$).

\item $\gamma$ denotes the \textit{cycle index} involved in the coarse-grid correction steps.

\end{itemize}

Given an initial guess $\mathbf{u}_{k}^{(0)}\in\mathbb{R}^{n_{k}}$, the standard multigrid scheme for solving the linear system $A_{k}\mathbf{u}_{k}=\mathbf{f}_{k}$ (with $\mathbf{f}_{k}\in\mathbb{R}^{n_{k}}$) can be described by~\cref{alg:MG}. The symbol $\text{MG}^{\gamma}$ in~\cref{alg:MG} means that the multigrid scheme will be carried out $\gamma$ iterations. In particular, $\gamma=1$ corresponds to the V-cycle and $\gamma=2$ to the W-cycle.

\begin{algorithm}[!htbp]

\caption{\ \textbf{Multigrid method at level $k$}: $\mathbf{u}_{\rm IMG}\gets\textbf{MG}\big(k, A_{k}, \mathbf{f}_{k}, \mathbf{u}_{k}^{(0)}\big)$} \label{alg:MG}

\smallskip

\begin{algorithmic}[1]
	
\STATE Presmoothing: $\mathbf{u}_{k}^{(1)}\gets\mathbf{u}_{k}^{(0)}+M_{k}^{-1}\big(\mathbf{f}_{k}-A_{k}\mathbf{u}_{k}^{(0)}\big)$

\smallskip

\STATE Restriction: $\mathbf{r}_{k-1}\gets P_{k}^{T}\big(\mathbf{f}_{k}-A_{k}\mathbf{u}_{k}^{(1)}\big)$

\smallskip

\STATE Coarse-grid correction: $\mathbf{e}_{k-1}\gets\begin{cases}
\hat{A}_{0}^{-1}\mathbf{r}_{0} & \text{if $k=1$}, \\
\textbf{MG}^{\gamma}\big(k-1, A_{k-1}, \mathbf{r}_{k-1}, \mathbf{0}\big) & \text{if $k>1$}.
\end{cases}$

\smallskip

\STATE Prolongation: $\mathbf{u}_{k}^{(2)}\gets\mathbf{u}_{k}^{(1)}+P_{k}\mathbf{e}_{k-1}$

\smallskip

\STATE Postsmoothing: $\mathbf{u}_{\rm IMG}\gets\mathbf{u}_{k}^{(2)}+M_{k}^{-T}\big(\mathbf{f}_{k}-A_{k}\mathbf{u}_{k}^{(2)}\big)$

\smallskip

\end{algorithmic}

\end{algorithm}

The iteration matrix of~\cref{alg:MG} is
\begin{equation}\label{tE-MGk1}
E_{\rm IMG}^{(k)}=\big(I-M_{k}^{-T}A_{k}\big)\Big[I-P_{k}\Big(I-\big(E_{\rm IMG}^{(k-1)}\big)^{\gamma}\Big)A_{k-1}^{-1}P_{k}^{T}A_{k}\Big]\big(I-M_{k}^{-1}A_{k}\big),
\end{equation}
which satisfies
\begin{displaymath}
\mathbf{u}_{k}-\mathbf{u}_{\rm IMG}=E_{\rm IMG}^{(k)}\big(\mathbf{u}_{k}-\mathbf{u}_{k}^{(0)}\big).
\end{displaymath}
In particular,
\begin{displaymath}
E_{\rm IMG}^{(1)}=\big(I-M_{1}^{-T}A_{1}\big)\big(I-P_{1}\hat{A}_{0}^{-1}P_{1}^{T}A_{1}\big)\big(I-M_{1}^{-1}A_{1}\big).
\end{displaymath}
By~\cref{tE-MGk1}, we have
\begin{displaymath}
A_{k}^{\frac{1}{2}}E_{\rm IMG}^{(k)}A_{k}^{-\frac{1}{2}}=N_{k}^{T}\Big[I-A_{k}^{\frac{1}{2}}P_{k}A_{k-1}^{-\frac{1}{2}}\Big(I-\big(A_{k-1}^{\frac{1}{2}}E_{\rm IMG}^{(k-1)}A_{k-1}^{-\frac{1}{2}}\big)^{\gamma}\Big)A_{k-1}^{-\frac{1}{2}}P_{k}^{T}A_{k}^{\frac{1}{2}}\Big]N_{k}
\end{displaymath}
with
\begin{displaymath}
N_{k}=I-A_{k}^{\frac{1}{2}}M_{k}^{-1}A_{k}^{\frac{1}{2}}.
\end{displaymath}
By induction, one can show that $A_{k}^{\frac{1}{2}}E_{\rm IMG}^{(k)}A_{k}^{-\frac{1}{2}}$ is symmetric and
\begin{displaymath}
\lambda\big(E_{\rm IMG}^{(k)}\big)=\lambda\Big(A_{k}^{\frac{1}{2}}E_{\rm IMG}^{(k)}A_{k}^{-\frac{1}{2}}\Big)\subset[0,1) \quad \forall\,k=1,\ldots,L,
\end{displaymath}
which lead to
\begin{displaymath}
\big\|E_{\rm IMG}^{(k)}\big\|_{A_{k}}<1.
\end{displaymath}
As a result, $E_{\rm IMG}^{(k)}$ can be expressed as
\begin{equation}\label{tE-MGk2}
E_{\rm IMG}^{(k)}=I-B_{k}^{-1}A_{k},
\end{equation}
where $B_{k}\in\mathbb{R}^{n_{k}\times n_{k}}$ is SPD and $B_{k}-A_{k}$ is SPSD. Combining~\cref{tE-MGk1} and~\cref{tE-MGk2}, we can obtain the recursive relation
\begin{displaymath}
B_{k}^{-1}=\overline{M}_{k}^{-1}+\big(I-M_{k}^{-T}A_{k}\big)P_{k}\Bigg(\sum_{j=0}^{\gamma-1}\big(I-B_{k-1}^{-1}A_{k-1}\big)^{j}\Bigg)B_{k-1}^{-1}P_{k}^{T}\big(I-A_{k}M_{k}^{-1}\big),
\end{displaymath}
where
\begin{equation}\label{barMk}
\overline{M}_{k}:=M_{k}\big(M_{k}+M_{k}^{T}-A_{k}\big)^{-1}M_{k}^{T}.
\end{equation}
Interchanging the roles of $M_{k}$ and $M_{k}^{T}$ in~\cref{barMk} yields another symmetrized smoother:
\begin{equation}\label{tildMk}
\widetilde{M}_{k}:=M_{k}^{T}\big(M_{k}+M_{k}^{T}-A_{k}\big)^{-1}M_{k}.
\end{equation}
It is easy to verify that both $\overline{M}_{k}-A_{k}$ and $\widetilde{M}_{k}-A_{k}$ are SPSD.

Comparing~\cref{tE-MGk1} with~\cref{tE_TG1}, we can see that~\cref{alg:MG} is essentially an inexact two-grid method with $M=M_{k}$, $A=A_{k}$, $P=P_{k}$, and
\begin{equation}\label{Bc}
B_{\rm c}=A_{k-1}\Big(I-\big(E_{\rm IMG}^{(k-1)}\big)^{\gamma}\Big)^{-1}.
\end{equation}

\begin{remark}
From~\cref{Bc}, we have
\begin{align*}
B_{\rm c}&=A_{k-1}^{\frac{1}{2}}\Big(A_{k-1}^{-\frac{1}{2}}-\big(E_{\rm IMG}^{(k-1)}\big)^{\gamma}A_{k-1}^{-\frac{1}{2}}\Big)^{-1}\\
&=A_{k-1}^{\frac{1}{2}}\Big(A_{k-1}^{-\frac{1}{2}}-A_{k-1}^{-\frac{1}{2}}A_{k-1}^{\frac{1}{2}}\big(E_{\rm IMG}^{(k-1)}\big)^{\gamma}A_{k-1}^{-\frac{1}{2}}\Big)^{-1}\\
&=A_{k-1}^{\frac{1}{2}}\Big(I-A_{k-1}^{\frac{1}{2}}\big(E_{\rm IMG}^{(k-1)}\big)^{\gamma}A_{k-1}^{-\frac{1}{2}}\Big)^{-1}A_{k-1}^{\frac{1}{2}}\\
&=A_{k-1}^{\frac{1}{2}}\Big[I-\Big(A_{k-1}^{\frac{1}{2}}E_{\rm IMG}^{(k-1)}A_{k-1}^{-\frac{1}{2}}\Big)^{\gamma}\Big]^{-1}A_{k-1}^{\frac{1}{2}}.
\end{align*}
Due to the fact that $A_{k-1}^{\frac{1}{2}}E_{\rm IMG}^{(k-1)}A_{k-1}^{-\frac{1}{2}}$ is symmetric and $\lambda\big(E_{\rm IMG}^{(k-1)}\big)\subset[0,1)$, $B_{\rm c}$ given by~\cref{Bc} is SPD.
\end{remark}

Define
\begin{align*}
\sigma_{\rm TG}^{(k)}:=\big\|E_{\rm TG}^{(k)}\big\|_{A_{k}} \quad \text{and} \quad \sigma_{\rm IMG}^{(k)}:=\big\|E_{\rm IMG}^{(k)}\big\|_{A_{k}},
\end{align*}
which are the convergence factors of the (exact) two-grid method and (inexact) multigrid method at level $k$, respectively. In view of~\cref{r12} and~\cref{Bc}, we have
\begin{align*}
r_{1}&=\lambda_{\min}\Big(I-\big(E_{\rm IMG}^{(k-1)}\big)^{\gamma}\Big)=1-\Big(\lambda_{\max}\big(E_{\rm IMG}^{(k-1)}\big)\Big)^{\gamma}=1-\big(\sigma_{\rm IMG}^{(k-1)}\big)^{\gamma},\\
r_{2}&=\lambda_{\max}\Big(I-\big(E_{\rm IMG}^{(k-1)}\big)^{\gamma}\Big)=1-\Big(\lambda_{\min}\big(E_{\rm IMG}^{(k-1)}\big)\Big)^{\gamma}\leq 1.
\end{align*}
Using~\cref{iTG1.1}, we obtain
\begin{displaymath}
\sigma_{\rm IMG}^{(k)}\leq 1-\frac{1}{K_{\rm TG}^{(k)}}+\big(\sigma_{\rm IMG}^{(k-1)}\big)^{\gamma}\Bigg(\frac{1}{K_{\rm TG}^{(k)}}-\lambda_{\min}(\widetilde{M}_{k}^{-1}A_{k})\Bigg),
\end{displaymath}
where
\begin{equation}\label{K-TG-k}
K_{\rm TG}^{(k)}=\max_{\mathbf{v}_{k}\in\mathbb{R}^{n_{k}}\backslash\{0\}}\frac{\big\|\big(I-\varPi_{\widetilde{M}_{k}}\big)\mathbf{v}_{k}\big\|_{\widetilde{M}_{k}}^{2}}{\|\mathbf{v}_{k}\|_{A_{k}}^{2}} \ \ \text{with} \ \ \varPi_{\widetilde{M}_{k}}=P_{k}(P_{k}^{T}\widetilde{M}_{k}P_{k})^{-1}P_{k}^{T}\widetilde{M}_{k}.
\end{equation}
It follows that
\begin{equation}\label{upper}
\sigma_{\rm IMG}^{(k)}\leq\sigma_{\rm TG}^{(k)}+\big(\sigma_{\rm IMG}^{(k-1)}\big)^{\gamma}\Big(1-\sigma_{\rm TG}^{(k)}-\lambda_{\min}(\widetilde{M}_{k}^{-1}A_{k})\Big),
\end{equation}
where we have used the fact $\sigma_{\rm TG}^{(k)}=1-\frac{1}{K_{\rm TG}^{(k)}}$.

\begin{remark}
The lower bound in~\cref{iTG1.1} yields
\begin{equation}\label{lower}
\sigma_{\rm IMG}^{(k)}\geq\sigma_{\rm TG}^{(k)}.
\end{equation}
Thus, a well converged multigrid method entails that the corresponding (exact) two-grid method has a fast convergence speed.
\end{remark}

Define
\begin{align}
\sigma_{\s L}&:=\max_{1\leq k\leq L}\sigma_{\rm TG}^{(k)},\label{sigL}\\
\varepsilon_{\s L}&:=\min_{1\leq k\leq L}\lambda_{\min}(\widetilde{M}_{k}^{-1}A_{k}).\label{epsL}
\end{align}
In view of~\cref{K-TG-k} and~\cref{epsL}, we have
\begin{displaymath}
K_{\rm TG}^{(k)}=\lambda_{\max}\big(A_{k}^{-1}\widetilde{M}_{k}\big(I-\varPi_{\widetilde{M}_{k}}\big)\big)\leq\lambda_{\max}(A_{k}^{-1}\widetilde{M}_{k})=\frac{1}{\lambda_{\min}(\widetilde{M}_{k}^{-1}A_{k})}\leq\frac{1}{\varepsilon_{\s L}}.
\end{displaymath}
Then
\begin{displaymath}
\sigma_{\rm TG}^{(k)}=1-\frac{1}{K_{\rm TG}^{(k)}}\leq 1-\varepsilon_{\s L} \quad \forall\,k=1,\ldots,L,
\end{displaymath}
and hence
\begin{displaymath}
0\leq\sigma_{\s L}\leq 1-\varepsilon_{\s L}.
\end{displaymath}
We remark that the extreme cases $\sigma_{\s L}=0$ and $\sigma_{\s L}=1-\varepsilon_{\s L}$ seldom occur in practice. In what follows, we only consider the nontrivial case
\begin{equation}
0<\sigma_{\s L}<1-\varepsilon_{\s L}.
\end{equation}

To analyze the convergence of~\cref{alg:MG}, we first prove a technical lemma.

\begin{lemma}\label{lem:MG}
Let $\sigma_{\s L}$ and $\varepsilon_{\s L}$ be defined by~\cref{sigL} and~\cref{epsL}, respectively. Then
\begin{equation}\label{x-gam}
(1-\sigma_{\s L}-\varepsilon_{\s L})x^{\gamma}-x+\sigma_{\s L}=0 \quad (0<x<1)
\end{equation}
has a unique root $x_{\gamma}$ in $\big(\sigma_{\s L},\frac{\sigma_{\s L}}{\sigma_{\s L}+\varepsilon_{\s L}}\big]$, and $\{x_{\gamma}\}_{\gamma=1}^{+\infty}$ is a strictly decreasing sequence with limit $\sigma_{\s L}$.
\end{lemma}

\begin{proof}
When $\gamma=1$, one can readily see that $\frac{\sigma_{\s L}}{\sigma_{\s L}+\varepsilon_{\s L}}$ is the root of~\cref{x-gam}. Let
\begin{displaymath}
F_{\gamma}(x)=(1-\sigma_{\s L}-\varepsilon_{\s L})x^{\gamma}-x+\sigma_{\s L} \quad (\gamma\geq 2).
\end{displaymath}
Then
\begin{displaymath}
\frac{\mathrm{d}F_{\gamma}(x)}{\mathrm{d}x}=\gamma(1-\sigma_{\s L}-\varepsilon_{\s L})x^{\gamma-1}-1.
\end{displaymath}

\begin{itemize}

\item If $\gamma(1-\sigma_{\s L}-\varepsilon_{\s L})\leq 1$, then $\frac{\mathrm{d}F_{\gamma}(x)}{\mathrm{d}x}<0$ in $(0,1)$, that is, $F_{\gamma}(x)$ is a strictly decreasing function in $(0,1)$. Due to
\begin{align*}
F_{\gamma}(\sigma_{\s L})&=(1-\sigma_{\s L}-\varepsilon_{\s L})\sigma_{\s L}^{\gamma}>0,\\
F_{\gamma}\bigg(\frac{\sigma_{\s L}}{\sigma_{\s L}+\varepsilon_{\s L}}\bigg)&<(1-\sigma_{\s L}-\varepsilon_{\s L})\frac{\sigma_{\s L}}{\sigma_{\s L}+\varepsilon_{\s L}}-\frac{\sigma_{\s L}}{\sigma_{\s L}+\varepsilon_{\s L}}+\sigma_{\s L}=0,
\end{align*}
it follows that $F_{\gamma}(x)=0$ has a unique root $x_{\gamma}$ in $\big(\sigma_{\s L},\frac{\sigma_{\s L}}{\sigma_{\s L}+\varepsilon_{\s L}}\big)$.

\item If $\gamma(1-\sigma_{\s L}-\varepsilon_{\s L})>1$, then
\begin{displaymath}
\begin{cases}
\frac{\mathrm{d}F_{\gamma}(x)}{\mathrm{d}x}<0 & \text{if} \ \, 0<x<\big(\gamma(1-\sigma_{\s L}-\varepsilon_{\s L})\big)^{\frac{1}{1-\gamma}},\\[2pt]
\frac{\mathrm{d}F_{\gamma}(x)}{\mathrm{d}x}>0 & \text{if} \ \, \big(\gamma(1-\sigma_{\s L}-\varepsilon_{\s L})\big)^{\frac{1}{1-\gamma}}<x<1.
\end{cases}
\end{displaymath}
The existence and uniqueness of $x_{\gamma}\in\big(\sigma_{\s L},\frac{\sigma_{\s L}}{\sigma_{\s L}+\varepsilon_{\s L}}\big)$ follow immediately from the facts $F_{\gamma}(\sigma_{\s L})>0$, $F_{\gamma}\big(\frac{\sigma_{\s L}}{\sigma_{\s L}+\varepsilon_{\s L}}\big)<0$, and $F_{\gamma}(1)<0$.

\end{itemize}

Since $x_{\gamma}<1$, it holds that
\begin{displaymath}
F_{\gamma+1}(x_{\gamma})=(1-\sigma_{\s L}-\varepsilon_{\s L})x_{\gamma}^{\gamma+1}-x_{\gamma}+\sigma_{\s L}<F_{\gamma}(x_{\gamma})=0,
\end{displaymath}
which, together with $F_{\gamma+1}(\sigma_{\s L})>0$, yields
\begin{displaymath}
\sigma_{\s L}<x_{\gamma+1}<x_{\gamma}.
\end{displaymath}
In addition, we deduce from $F_{\gamma}(x_{\gamma})=0$ that
\begin{displaymath}
x_{\gamma}=(1-\sigma_{\s L}-\varepsilon_{\s L})x_{\gamma}^{\gamma}+\sigma_{\s L},
\end{displaymath}
which leads to
\begin{displaymath}
\lim\limits_{\gamma\rightarrow+\infty}x_{\gamma}=\sigma_{\s L}.
\end{displaymath}
This completes the proof.
\end{proof}

Using~\cref{iTG1.1}, \cref{upper}, and~\cref{lem:MG}, we can obtain the following convergence estimate.

\begin{theorem}\label{thm:MG1}
Let $\sigma_{\s L}$ and $\varepsilon_{\s L}$ be defined by~\cref{sigL} and~\cref{epsL}, respectively. Let $x_{\gamma}$ be the (unique) root of~\cref{x-gam} contained in $\big(\sigma_{\s L},\frac{\sigma_{\s L}}{\sigma_{\s L}+\varepsilon_{\s L}}\big]$. If
\begin{equation}\label{condMG1}
\lambda\big(\hat{A}_{0}^{-1}A_{0}\big)\subset\bigg[\frac{1-\varepsilon_{\s L}-x_{\gamma}}{1-\varepsilon_{\s L}-\sigma_{\s L}},\,1\bigg],
\end{equation}
then
\begin{equation}\label{estMG1}
\sigma_{\rm IMG}^{(k)}\leq x_{\gamma} \quad \forall\,k=1,\ldots,L.
\end{equation}
\end{theorem}

\begin{proof}
By~\cref{iTG1.1} and~\cref{condMG1}, we have
\begin{align*}
\sigma_{\rm IMG}^{(1)}&\leq 1-\frac{1-\varepsilon_{\s L}-x_{\gamma}}{(1-\varepsilon_{\s L}-\sigma_{\s L})K_{\rm TG}^{(1)}}-\frac{x_{\gamma}-\sigma_{\s L}}{1-\varepsilon_{\s L}-\sigma_{\s L}}\lambda_{\min}(\widetilde{M}_{1}^{-1}A_{1})\\
&=1-\frac{(1-\varepsilon_{\s L}-x_{\gamma})\big(1-\sigma_{\rm TG}^{(1)}\big)+(x_{\gamma}-\sigma_{\s L})\lambda_{\min}(\widetilde{M}_{1}^{-1}A_{1})}{1-\varepsilon_{\s L}-\sigma_{\s L}}\\
&\leq 1-\frac{(1-\varepsilon_{\s L}-x_{\gamma})\Big(1-\max\limits_{1\leq k\leq L}\sigma_{\rm TG}^{(k)}\Big)+(x_{\gamma}-\sigma_{\s L})\min\limits_{1\leq k\leq L}\lambda_{\min}(\widetilde{M}_{k}^{-1}A_{k})}{1-\varepsilon_{\s L}-\sigma_{\s L}}\\
&=1-\frac{(1-\varepsilon_{\s L}-x_{\gamma})(1-\sigma_{\s L})+(x_{\gamma}-\sigma_{\s L})\varepsilon_{\s L}}{1-\varepsilon_{\s L}-\sigma_{\s L}}=x_{\gamma}.
\end{align*}
From~\cref{upper}, we have
\begin{align*}
\sigma_{\rm IMG}^{(k)}&\leq\Big(1-\big(\sigma_{\rm IMG}^{(k-1)}\big)^{\gamma}\Big)\sigma_{\rm TG}^{(k)}+\big(\sigma_{\rm IMG}^{(k-1)}\big)^{\gamma}\big(1-\lambda_{\min}(\widetilde{M}_{k}^{-1}A_{k})\big)\\
&\leq\Big(1-\big(\sigma_{\rm IMG}^{(k-1)}\big)^{\gamma}\Big)\max_{1\leq k\leq L}\sigma_{\rm TG}^{(k)}+\big(\sigma_{\rm IMG}^{(k-1)}\big)^{\gamma}\Big(1-\min_{1\leq k\leq L}\lambda_{\min}(\widetilde{M}_{k}^{-1}A_{k})\Big)\\
&=\Big(1-\big(\sigma_{\rm IMG}^{(k-1)}\big)^{\gamma}\Big)\sigma_{\s L}+(1-\varepsilon_{\s L})\big(\sigma_{\rm IMG}^{(k-1)}\big)^{\gamma}\\
&=\sigma_{\s L}+(1-\sigma_{\s L}-\varepsilon_{\s L})\big(\sigma_{\rm IMG}^{(k-1)}\big)^{\gamma}.
\end{align*}
If $\sigma_{\rm IMG}^{(k-1)}\leq x_{\gamma}$, then
\begin{displaymath}
\sigma_{\rm IMG}^{(k)}\leq(1-\sigma_{\s L}-\varepsilon_{\s L})x_{\gamma}^{\gamma}+\sigma_{\s L}=F_{\gamma}(x_{\gamma})+x_{\gamma}=x_{\gamma}.
\end{displaymath}
The estimate~\cref{estMG1} then follows by induction.
\end{proof}

The following corollary particularizes~\cref{thm:MG1} for the cases $\gamma=1$ and $\gamma=2$.

\begin{corollary}\label{cor:MG1}
Under the assumptions of~\cref{thm:MG1}, it holds that
\begin{displaymath}
\sigma_{\rm IMG}^{(k)}\leq\begin{cases}
\frac{\sigma_{\s L}}{\sigma_{\s L}+\varepsilon_{\s L}} &\text{if $\gamma=1$},\\[2pt]
\frac{2\sigma_{\s L}}{1+\sqrt{(1-2\sigma_{\s L})^{2}+4\sigma_{\s L}\varepsilon_{\s L}}} &\text{if $\gamma=2$},
\end{cases} \quad \forall\,k=1,\ldots,L.
\end{displaymath}
\end{corollary}

\begin{remark}
For the V-cycle multigrid methods, if $\sigma_{\s L}\leq C\varepsilon_{\s L}$, then
\begin{displaymath}
\sigma_{\rm IMG}^{(k)}\leq\frac{\frac{\sigma_{\s L}}{\varepsilon_{\s L}}}{\frac{\sigma_{\s L}}{\varepsilon_{\s L}}+1}\leq\frac{C}{C+1}.
\end{displaymath}
For the W-cycle multigrid methods, if $\sigma_{\s L}\leq\sigma<1$ for a level-independent quantity $\sigma$, we deduce from~\cref{cor:MG1} that
\begin{displaymath}
\sigma_{\rm IMG}^{(k)}\leq\frac{2\sigma_{\s L}}{1+\sqrt{(1-2\sigma_{\s L})^{2}+4\sigma_{\s L}\varepsilon_{\s L}}}<\frac{\sigma_{\s L}}{1-\sigma_{\s L}}\leq\frac{\sigma}{1-\sigma},
\end{displaymath}
that is, our result improves the existing one in~\cite[Theorem~3.1]{Notay2007}.
\end{remark}

The next theorem gives an upper bound for $\sigma_{\rm IMG}^{(k)}$ that depends on the level index $k$, which sharpens the bound in~\cref{estMG1}.

\begin{theorem}\label{thm:MG2}
Let $\sigma_{\s L}$ and $\varepsilon_{\s L}$ be defined by~\cref{sigL} and~\cref{epsL}, respectively. Let $x_{\gamma}$ be the (unique) root of~\cref{x-gam} contained in $\big(\sigma_{\s L},\frac{\sigma_{\s L}}{\sigma_{\s L}+\varepsilon_{\s L}}\big]$. If $\sigma_{\rm IMG}^{(1)}<x_{\gamma}$, then
\begin{equation}\label{estMG2}
\sigma_{\rm IMG}^{(k)}\leq x_{\gamma}-\big(x_{\gamma}-\sigma_{\rm IMG}^{(1)}\big)\Bigg((1-\sigma_{\s L}-\varepsilon_{\s L})x_{\gamma}^{\gamma-1}\sum_{j=0}^{\gamma-1}\bigg(\frac{\delta_{\s L}}{x_{\gamma}}\bigg)^{j}\Bigg)^{k-1},
\end{equation}
where
\begin{equation}\label{deltaL}
\delta_{\s L}:=\min_{1\leq k\leq L}\sigma_{\rm TG}^{(k)}.
\end{equation}
\end{theorem}

\begin{proof}
Similarly to the proof of~\cref{thm:MG1}, one can prove that $\sigma_{\rm IMG}^{(k)}<x_{\gamma}$ for all $k=1,\ldots,L$. Due to
\begin{displaymath}
\sigma_{\rm IMG}^{(k)}\leq\sigma_{\s L}+(1-\sigma_{\s L}-\varepsilon_{\s L})\big(\sigma_{\rm IMG}^{(k-1)}\big)^{\gamma} \quad \text{and} \quad x_{\gamma}=\sigma_{\s L}+(1-\sigma_{\s L}-\varepsilon_{\s L})x_{\gamma}^{\gamma},
\end{displaymath}
it follows that
\begin{displaymath}
x_{\gamma}-\sigma_{\rm IMG}^{(k)}\geq(1-\sigma_{\s L}-\varepsilon_{\s L})\Big(x_{\gamma}^{\gamma}-\big(\sigma_{\rm IMG}^{(k-1)}\big)^{\gamma}\Big),
\end{displaymath}
which yields
\begin{displaymath}
\frac{x_{\gamma}-\sigma_{\rm IMG}^{(k)}}{x_{\gamma}-\sigma_{\rm IMG}^{(k-1)}}\geq(1-\sigma_{\s L}-\varepsilon_{\s L})x_{\gamma}^{\gamma-1}\sum_{j=0}^{\gamma-1}\Bigg(\frac{\sigma_{\rm IMG}^{(k-1)}}{x_{\gamma}}\Bigg)^{j}.
\end{displaymath}
By~\cref{lower} and~\cref{deltaL}, we have
\begin{displaymath}
\sigma_{\rm IMG}^{(k-1)}\geq\sigma_{\rm TG}^{(k-1)}\geq\delta_{\s L}.
\end{displaymath}
Then
\begin{displaymath}
\frac{x_{\gamma}-\sigma_{\rm IMG}^{(k)}}{x_{\gamma}-\sigma_{\rm IMG}^{(k-1)}}\geq(1-\sigma_{\s L}-\varepsilon_{\s L})x_{\gamma}^{\gamma-1}\sum_{j=0}^{\gamma-1}\bigg(\frac{\delta_{\s L}}{x_{\gamma}}\bigg)^{j}.
\end{displaymath}
Hence,
\begin{displaymath}
\frac{x_{\gamma}-\sigma_{\rm IMG}^{(k)}}{x_{\gamma}-\sigma_{\rm IMG}^{(1)}}=\prod_{i=2}^{k}\frac{x_{\gamma}-\sigma_{\rm IMG}^{(i)}}{x_{\gamma}-\sigma_{\rm IMG}^{(i-1)}}\geq\Bigg((1-\sigma_{\s L}-\varepsilon_{\s L})x_{\gamma}^{\gamma-1}\sum_{j=0}^{\gamma-1}\bigg(\frac{\delta_{\s L}}{x_{\gamma}}\bigg)^{j}\Bigg)^{k-1},
\end{displaymath}
which leads to the estimate~\cref{estMG2}.
\end{proof}

\begin{remark}
Observe that
\begin{align*}
(1-\sigma_{\s L}-\varepsilon_{\s L})x_{\gamma}^{\gamma-1}\sum_{j=0}^{\gamma-1}\bigg(\frac{\delta_{\s L}}{x_{\gamma}}\bigg)^{j}&=\frac{(1-\sigma_{\s L}-\varepsilon_{\s L})x_{\gamma}^{\gamma}-(1-\sigma_{\s L}-\varepsilon_{\s L})\delta_{\s L}^{\gamma}}{x_{\gamma}-\delta_{\s L}}\\
&=\frac{x_{\gamma}-\sigma_{\s L}-(1-\sigma_{\s L}-\varepsilon_{\s L})\delta_{\s L}^{\gamma}}{x_{\gamma}-\delta_{\s L}}\in(0,1).
\end{align*}
Thus, the upper bound in~\cref{estMG2} is a strictly increasing function with respect to $k$.
\end{remark}

The condition $\sigma_{\rm IMG}^{(1)}<x_{\gamma}$ in~\cref{thm:MG2} will be satisfied if $\hat{A}_{0}$ is simply chosen as $A_{0}$, in which case the convergence factor is denoted by $\sigma_{\rm MG}^{(k)}$. This yields the following corollary.

\begin{corollary}
Let $\sigma_{\s L}$, $\varepsilon_{\s L}$, and $\delta_{\s L}$ be defined by~\cref{sigL}, \cref{epsL}, and~\cref{deltaL}, respectively. Let $x_{\gamma}$ be the (unique) root of~\cref{x-gam} contained in $\big(\sigma_{\s L},\frac{\sigma_{\s L}}{\sigma_{\s L}+\varepsilon_{\s L}}\big]$. If $\hat{A}_{0}=A_{0}$, then
\begin{equation}\label{estMG2-cor}
\sigma_{\rm MG}^{(k)}\leq x_{\gamma}-(x_{\gamma}-\sigma_{\s L})\Bigg((1-\sigma_{\s L}-\varepsilon_{\s L})x_{\gamma}^{\gamma-1}\sum_{j=0}^{\gamma-1}\bigg(\frac{\delta_{\s L}}{x_{\gamma}}\bigg)^{j}\Bigg)^{k-1}.
\end{equation}
In particular, one has
\begin{equation}\label{VW}
\sigma_{\rm MG}^{(k)}\leq\begin{cases}
x_{1}\big(1-(1-\sigma_{\s L}-\varepsilon_{\s L})^{k}\big) &\text{if $\gamma=1$},\\[2pt]
x_{2}-(x_{2}-\sigma_{\s L})\big((1-\sigma_{\s L}-\varepsilon_{\s L})(x_{2}+\delta_{\s L})\big)^{k-1} &\text{if $\gamma=2$},
\end{cases}
\end{equation}
where
\begin{displaymath}
x_{1}=\frac{\sigma_{\s L}}{\sigma_{\s L}+\varepsilon_{\s L}} \quad \text{and} \quad x_{2}=\frac{2\sigma_{\s L}}{1+\sqrt{(1-2\sigma_{\s L})^{2}+4\sigma_{\s L}\varepsilon_{\s L}}}.
\end{displaymath}
\end{corollary}

\begin{proof}
If $\hat{A}_{0}=A_{0}$, then
\begin{displaymath}
\sigma_{\rm MG}^{(1)}=\sigma_{\rm TG}^{(1)}\leq\sigma_{\s L}<x_{\gamma}.
\end{displaymath}
The estimate~\cref{estMG2-cor} then follows immediately from~\cref{estMG2}.
\end{proof}

\begin{remark}
For some simple smoothers, $\varepsilon_{\s L}$ may be very small. A solution is to use more powerful smoothers, like the SIF (structured incomplete factorization) and eSIF preconditioners in~\cite{Xia2017,Xia2020}. For the SIF-type smoothers, $\varepsilon_{\s L}$ is a controllable quantity, which will not be tiny if a reasonable truncation tolerance is used. Note that our theory is valid as long as $\sigma_{\s L}<1-\varepsilon_{\s L}$. If $\varepsilon_{\s L}$ is very small, then multigrid methods can carry over the convergence properties of two-grid methods under a very weak constraint on two-grid convergence speed. In the extreme case when $\varepsilon_{\s L}$ is zero, we deduce from~\cref{VW} that
\begin{itemize}

\item the V-cycle multigrid satisfies
\begin{equation}\label{V}
\sigma_{\rm MG}^{(k)}\leq 1-(1-\sigma_{\s L})^{k};
\end{equation}

\item the W-cycle multigrid satisfies
\begin{equation}\label{W}
\sigma_{\rm MG}^{(k)}\leq\begin{cases}
\frac{\sigma_{\s L}}{1-\sigma_{\s L}}-\frac{\sigma_{\s L}^{2}}{1-\sigma_{\s L}}\big(\sigma_{\s L}+(1-\sigma_{\s L})\delta_{\s L}\big)^{k-1} & \text{if $\sigma_{\s L}<\frac{1}{2}$},\\[2pt]
1-(1-\sigma_{\s L})^{k}(1+\delta_{\s L})^{k-1} & \text{if $\frac{1}{2}\leq\sigma_{\s L}<1$}.
\end{cases}
\end{equation}

\end{itemize}
We remark that the estimates~\cref{V} and~\cref{W} are applicable for $\sigma_{\s L}<1$.
\end{remark}

To compare the performances of~\cref{V}, \cref{W}, and the existing estimate in~\cite[Theorem~3.1]{Notay2007}, we give a numerical example: the 2D Poisson's equation with homogeneous Dirichlet boundary condition on a unit square (using the P1-finite element on a quasi-uniform grid with one million degrees of freedom). The resulting linear system is solved by the classical algebraic multigrid method~\cite{Brandt1985,Ruge1987} in a standard setting: the classical coarsening and the direct interpolation are exploited. In the experiments, we set the number of pre- and post-smoothing steps to be $1$, the finest-level index $L$ to be $5$, and the strong threshold to be $\frac{1}{4}$ (no truncation is applied). The coarsest-grid systems are solved by a sparse direct solver. The asymptotic convergence factor of a multigrid method is computed when the energy norm of error is less than $10^{-12}$.

\begin{table}[!htbp]
\caption{Actual convergence factors of multigrid methods and their upper bounds}
\centering
\begin{tabular}{c c c c c c c}
\hline\hline
\textbf{Smoother} & $\bm{\delta_{\s L}}$ & $\bm{\sigma_{\s L}}$ & \textbf{Cycle} & \textbf{Conv. factor} & \textbf{Existing} & \textbf{New} \\[0.5ex]
\hline
  &  &  & V & 0.876 & N/A & 0.955 \\[-1.2ex]
\raisebox{2ex}{Gauss--Seidel} & \raisebox{2ex}{0.232} & \raisebox{2ex}{0.462} & W & 0.556 & 0.859 & 0.812 \\[2ex]
  &  &  & V & 0.905 & N/A & 0.993 \\[-1.2ex]
\raisebox{2ex}{$\omega$-Jacobi ($\omega=0.5$)} & \raisebox{2ex}{0.292} & \raisebox{2ex}{0.625} & W & 0.639 & Fail & 0.979 \\
\hline
\end{tabular}
\label{tab:comp}
\end{table}

\cref{tab:comp} displays that our estimates improve the existing one in~\cite{Notay2007}. Moreover, the W-cycle multigrid method may carry over two-grid convergence even when $\sigma_{\s L}>\frac{1}{2}$.

\section{Conclusions} \label{sec:con}

In this paper, we present a theoretical framework for the convergence analysis of inexact two-grid methods (for SPD problems), which improves and extends the existing theory for two-grid methods. A natural question is how to construct the coarse-grid matrix $B_{\rm c}$ or approximate the Galerkin coarse-grid matrix $A_{\rm c}$, which serves as a motivation for designing new multigrid-based algorithms. As an application of the framework, we establish a unified convergence theory for standard multigrid methods, which allows the coarsest-grid system to be solved approximately. Furthermore, the framework can be used to analyze hybrid multilevel methods, like the VW- and WV-cycle multigrid methods in~\cite{XXF2021}. In the future, we expect to analyze the convergence of inexact two-grid methods for nonsymmetric problems, which is an interesting topic that deserves in-depth study; see, e.g.,~\cite{Manteuffel2018,Manteuffel2019,Manteuffel2019-2,Garcia2020,Notay2020}.

\section*{Acknowledgment}

The authors would like to thank the anonymous referees for their valuable comments and suggestions, which greatly improved the original version of this paper.

\bibliographystyle{siamplain}
\bibliography{references}

\begin{thebibliography}{10}

\bibitem{Axelsson1994}
{\sc O.~Axelsson}, {\em {Iterative Solution Methods}}, Cambridge University
  Press, Cambridge, 1994.

\bibitem{Bank1988}
{\sc R.~E. Bank, T.~F. Dupont, and H.~Yserentant}, {\em {The hierarchical basis
  multigrid method}}, Numer. Math., 52 (1988), pp.~427--458.

\bibitem{Braess1982}
{\sc D.~Braess}, {\em {The convergence rate of a multigrid method with
  Gauss--Seidel relaxation for the Poisson equation}}, in Multigrid Methods (W.
  Hackbusch and U. Trottenberg, eds), Lecture Notes in Mathematics, Springer,
  Berlin, Heidelberg, 960 (1982), pp.~368--386.

\bibitem{Braess1983}
{\sc D.~Braess and W.~Hackbusch}, {\em {A new convergence proof for the
  multigrid method including the V-cycle}}, SIAM J. Numer. Anal., 20 (1983),
  pp.~967--975.

\bibitem{Brandt1973}
{\sc A.~E. Brandt}, {\em {Multi-level adaptive technique (MLAT) for fast
  numerical solution to boundary value problems}}, Proceedings of the Third
  International Conference on Numerical Methods in Fluid Mechanics, Lecture
  Notes in Physics (H. Cabannes and R. Temam, eds), Springer, Berlin,
  Heidelberg, 18 (1973), pp.~82--89.

\bibitem{Brandt1977}
{\sc A.~E. Brandt}, {\em {Multi-level adaptive solutions to boundary-value
  problems}}, Math. Comp., 31 (1977), pp.~333--390.

\bibitem{Brandt1986}
{\sc A.~E. Brandt}, {\em {Algebraic multigrid theory: The symmetric case}},
  Appl. Math. Comput., 19 (1986), pp.~23--56.

\bibitem{Brandt2000}
{\sc A.~E. Brandt}, {\em {General highly accurate algebraic coarsening}},
  Electron. Trans. Numer. Anal., 10 (2000), pp.~1--20.

\bibitem{Brandt1985}
{\sc A.~E. Brandt, S.~F. McCormick, and J.~W. Ruge}, {\em {Algebraic multigrid
  (AMG) for sparse matrix equations}}, in Sparsity and Its Applications
  (Loughborough, 1983), Cambridge University Press, Cambridge,  (1985),
  pp.~257--284.

\bibitem{Brannick2018}
{\sc J.~Brannick, F.~Cao, K.~Kahl, R.~D. Falgout, and X.~Hu}, {\em {Optimal
  interpolation and compatible relaxation in classical algebraic multigrid}},
  SIAM J. Sci. Comput., 40 (2018), pp.~A1473--A1493.

\bibitem{Briggs2000}
{\sc W.~L. Briggs, V.~E. Henson, and S.~F. McCormick}, {\em {A Multigrid
  Tutorial}}, SIAM, Philadelphia, PA, 2nd~ed., 2000.

\bibitem{Falgout2014}
{\sc R.~D. Falgout and J.~B. Schroder}, {\em {Non-Galerkin coarse grids for
  algebraic multigrid}}, SIAM J. Sci. Comput., 36 (2014), pp.~C309--C334.

\bibitem{Falgout2004}
{\sc R.~D. Falgout and P.~S. Vassilevski}, {\em {On generalizing the algebraic
  multigrid framework}}, SIAM J. Numer. Anal., 42 (2004), pp.~1669--1693.

\bibitem{Falgout2005}
{\sc R.~D. Falgout, P.~S. Vassilevski, and L.~T. Zikatanov}, {\em {On two-grid
  convergence estimates}}, Numer. Linear Algebra Appl., 12 (2005),
  pp.~471--494.

\bibitem{Fedorenko1962}
{\sc R.~P. Fedorenko}, {\em {A relaxation method for solving elliptic
  difference equations}}, USSR Comp. Math. Math. Phys., 1 (1962),
  pp.~1092--1096.

\bibitem{Fedorenko1964}
{\sc R.~P. Fedorenko}, {\em {The speed of convergence of one iterative
  process}}, USSR Comp. Math. Math. Phys., 4 (1964), pp.~227--235.

\bibitem{Garcia2020}
{\sc L.~Garc\'{i}a~Ramos, R.~Kehl, and R.~Nabben}, {\em {Projections,
  deflation, and multigrid for nonsymmetric matrices}}, SIAM J. Matrix Anal.
  Appl., 41 (2020), pp.~83--105.

\bibitem{Hackbusch1980}
{\sc W.~Hackbusch}, {\em {Convergence of multi-grid iterations applied to
  difference equations}}, Math. Comp., 34 (1980), pp.~425--440.

\bibitem{Hackbusch1981}
{\sc W.~Hackbusch}, {\em {On the convergence of multi-grid iterations}},
  Beitr\"{a}ge Numer. Math., 9 (1981), pp.~213--239.

\bibitem{Hackbusch1985}
{\sc W.~Hackbusch}, {\em {Multi-Grid Methods and Applications}},
  Springer-Verlag, Berlin, Heidelberg, 1985.

\bibitem{Horn2013}
{\sc R.~A. Horn and C.~R. Johnson}, {\em {Matrix Analysis}}, Cambridge
  University Press, Cambridge, 2nd~ed., 2013.

\bibitem{MacLachlan2014}
{\sc S.~P. MacLachlan and L.~N. Olson}, {\em {Theoretical bounds for algebraic
  multigrid performance: review and analysis}}, Numer. Linear Algebra Appl., 21
  (2014), pp.~194--220.

\bibitem{Mandel1988}
{\sc J.~Mandel, S.~F. McCormick, and J.~W. Ruge}, {\em {An algebraic theory for
  multigrid methods for variational problems}}, SIAM J. Numer. Anal., 25
  (1988), pp.~91--110.

\bibitem{Manteuffel2019}
{\sc T.~A. Manteuffel, S.~M\"{u}nzenmaier, J.~W. Ruge, and B.~S. Southworth},
  {\em {Nonsymmetric reduction-based algebraic multigrid}}, SIAM J. Sci.
  Comput., 41 (2019), pp.~S242--S268.

\bibitem{Manteuffel2017}
{\sc T.~A. Manteuffel, L.~N. Olson, J.~B. Schroder, and B.~S. Southworth}, {\em
  {A root-node--based algebraic multigrid method}}, SIAM J. Sci. Comput., 39
  (2017), pp.~S723--S756.

\bibitem{Manteuffel2018}
{\sc T.~A. Manteuffel, J.~W. Ruge, and B.~S. Southworth}, {\em {Nonsymmetric
  algebraic multigrid based on local approximate ideal restriction
  ($\ell$AIR)}}, SIAM J. Sci. Comput., 40 (2018), pp.~A4105--A4130.

\bibitem{Manteuffel2019-2}
{\sc T.~A. Manteuffel and B.~S. Southworth}, {\em {Convergence in norm of
  nonsymmetric algebraic multigrid}}, SIAM J. Sci. Comput., 41 (2019),
  pp.~S269--S296.

\bibitem{McCormick1985}
{\sc S.~F. McCormick}, {\em {Multigrid methods for variational problems:
  General theory for the V-cycle}}, SIAM J. Numer. Anal., 22 (1985),
  pp.~634--643.

\bibitem{Notay2007}
{\sc Y.~Notay}, {\em {Convergence analysis of perturbed two-grid and multigrid
  methods}}, SIAM J. Numer. Anal., 45 (2007), pp.~1035--1044.

\bibitem{Notay2015}
{\sc Y.~Notay}, {\em {Algebraic theory of two-grid methods}}, Numer. Math.
  Theor. Meth. Appl., 8 (2015), pp.~168--198.

\bibitem{Notay2020}
{\sc Y.~Notay}, {\em {Analysis of two-grid methods: The nonnormal case}}, Math.
  Comp., 89 (2020), pp.~807--827.

\bibitem{Oswald1994}
{\sc P.~Oswald}, {\em {Multilevel Finite Element Approximation: Theory and
  Applications}}, Teubner Skripten zur Numerik, Vieweg+Teubner Verlag,
  Wiesbaden, 1994.

\bibitem{Ruge1987}
{\sc J.~W. Ruge and K.~St\"{u}ben}, {\em {Algebraic multigrid}}, in Multigrid
  Methods, Frontiers Appl. Math., SIAM, Philadelphia, 3 (1987), pp.~73--130.

\bibitem{Sterck2008}
{\sc H.~D. Sterck, R.~D. Falgout, J.~W. Nolting, and U.~M. Yang}, {\em
  {Distance-two interpolation for parallel algebraic multigrid}}, Numer. Linear
  Algebra Appl., 15 (2008), pp.~115--139.

\bibitem{Sterck2006}
{\sc H.~D. Sterck, U.~M. Yang, and J.~J. Heys}, {\em {Reducing complexity in
  parallel algebraic multigrid preconditioners}}, SIAM J. Matrix Anal. Appl.,
  27 (2006), pp.~1019--1039.

\bibitem{Trottenberg2001}
{\sc U.~Trottenberg, C.~W. Oosterlee, and A.~Sch\"{u}ller}, {\em {Multigrid}},
  Academic Press, New York, 2001.

\bibitem{Vassilevski2008}
{\sc P.~S. Vassilevski}, {\em {Multilevel Block Factorization Preconditioners:
  Matrix-based Analysis and Algorithms for Solving Finite Element Equations}},
  Springer-Verlag, New York, 2008.

\bibitem{Xia2020}
{\sc J.~Xia}, {\em {Robust and effective eSIF preconditioning for general SPD
  matrices}}, \rm arXiv:2007.03729,  (2020).

\bibitem{Xia2017}
{\sc J.~Xia and Z.~Xin}, {\em {Effective and robust preconditioning of general
  SPD matrices via structured incomplete factorization}}, SIAM J. Matrix Anal.
  Appl., 38 (2017), pp.~1298--1322.

\bibitem{Xu1992}
{\sc J.~Xu}, {\em {Iterative methods by space decomposition and subspace
  correction}}, SIAM Rev., 34 (1992), pp.~581--613.

\bibitem{XZ2002}
{\sc J.~Xu and L.~T. Zikatanov}, {\em {The method of alternating projections
  and the method of subspace corrections in Hilbert space}}, J. Amer. Math.
  Soc., 15 (2002), pp.~573--597.

\bibitem{XZ2017}
{\sc J.~Xu and L.~T. Zikatanov}, {\em {Algebraic multigrid methods}}, Acta
  Numer., 26 (2017), pp.~591--721.

\bibitem{XXF-thesis}
{\sc X.~Xu}, {\em {Algebraic theory of multigrid methods}}, Ph.D.~Thesis (in
  Chinese), University of Chinese Academy of Sciences,  (2019).

\bibitem{XXF2018}
{\sc X.~Xu and C.-S. Zhang}, {\em {On the ideal interpolation operator in
  algebraic multigrid methods}}, SIAM J. Numer. Anal., 56 (2018),
  pp.~1693--1710.

\bibitem{XXF2021}
{\sc X.~Xu and C.-S. Zhang}, {\em {Convergence analysis of multigrid methods
  with alternating cycles}}, \rm submitted,  (2021).

\bibitem{Yserentant1993}
{\sc H.~Yserentant}, {\em {Old and new convergence proofs for multigrid
  methods}}, Acta Numer., 2 (1993), pp.~285--326.

\bibitem{Zikatanov2008}
{\sc L.~T. Zikatanov}, {\em {Two-sided bounds on the convergence rate of
  two-level methods}}, Numer. Linear Algebra Appl., 15 (2008), pp.~439--454.

\end{thebibliography}

\end{document}